%% file: ch2cross.tex
\documentclass[12pt, a4paper]{article}
\usepackage{epsfig}
\usepackage{times,amsmath,amssymb,graphicx,subfigure, xcolor}
\usepackage{pslatex,a4wide,xspace}
\usepackage{amsthm,rotating}
\usepackage{amstext}
\usepackage{amsopn}

\usepackage{latexsym, graphicx}

\newtheorem{theorem}{Theorem}

\newtheorem{lemma}[theorem]{Lemma}

\newtheorem{claim}{Claim}[theorem]

\theoremstyle{definition}    
\newtheorem{definition}[theorem]{Definition}

\DeclareMathOperator{\ch}{ch}
\DeclareMathOperator{\cro}{cr}

\begin{document}

\title{ 5-choosability of graphs with 2 crossings
\thanks{This work was partially supported by Equipe Associ\'ee EWIN.}}

\author{
Victor Campos\footnote{ Universidade Federal do Cear\'a, Departamento de Computa\c{c}ao,
Bloco 910, Campus do Pici, Fortaleza, Cear\'a, CEP 60455-760, Brasil.
{\tt campos@lia.ufc.br}; Partially supported by CNPq/Brazil.
}
\and
Fr\'ed\'eric Havet
\footnote{Projet Mascotte, I3S(CNRS, UNSA) and INRIA,
2004 route des lucioles, BP 93, 06902 Sophia-Antipolis Cedex, France.
{\tt Frederic.Havet@inria.fr}; Partially supported by the ANR Blanc International ANR-09-blan-0373-01.}  
}

\date{\today}

\maketitle

\noindent
{\bf Keywords:} List colouring; Choosability; Crossing number \\


\begin{abstract}
We show that every graph with two crossings is $5$-choosable.
We also prove that every graph which can be made planar by removing one edge is $5$-choosable.
\end{abstract}

%
%
%
%

\section{Introduction}

The crossing number of a graph $G$, denoted by $\cro(G)$, is the minimum number of crossings
in any drawing of $G$ in the plane.

The Four Colour Theorem states that, if a graph has crossing number zero (i.e. is planar), then it is $4$-colourable.
Deleting one vertex per crossing, it follows that $\chi(G)\leq 4 +cr(G)$.
So it is natural to ask for the smallest integer $f(k)$ such that every graph $G$ with crossing number at most $k$ is $f(k)$-colourable?
Settling a conjecture of Albertson~\cite{Alb08}, 
Schaefer~\cite{Sch} showed that $f(k)=O\left (k^{1/4}\right )$. This upper bound is tight up to a constant factor since $\chi(K_n)=n$ and
$\cro(K_n)\leq {|E(K_n)| \choose 2} =  \binom{{n\choose 2}}{2}\leq \frac{1}{8}n^4$.

The values of $f(k)$ are known for a number of small values of $k$.
The Four Colour Theorem states $f(0)=4$ and implies easily that $f(1)\leq 5$.
Since $\cro(K_5)=1$, we have $f(1)=5$.
Oporowski and Zhao~\cite{OpZh} showed that $f(2)=5$.
Since $\cro(K_6)=3$, we have $f(3)=6$. Further, Albertson et al. \cite{AHMW}
showed that $f(6)=6$. Albertson then conjectured that if $\chi(G)=r$, then $\cro(G)\leq \cro(K_r)$.
This conjecture was proved by Bar\'at and T\'oth~\cite{BaTo10} for $r\leq 16$.

A \emph{list assignment} of a graph $G$ is a function $L$ that assigns to each vertex 
$v \in V(G)$ a list $L(v)$ of available colours. An \emph{$L$-colouring} is a function 
$\varphi: V(G) \rightarrow \bigcup_v L(v)$ such 
that $\varphi(v) \in L(v)$ for every $v \in V(G)$ and
$\varphi(u) \neq \varphi(v)$ whenever $u$ and $v$ are adjacent vertices 
of $G$. If $G$ admits an $L$-colouring, then it is \emph{$L$-colourable}. 
A graph $G$ is \emph{$k$-choosable} if it is $L$-colourable for
every list assignment $L$ such that $|L(v)| \ge k$ for all $v \in V(G)$.
The \emph{choose number} of $G$, denoted by  $\ch(G)$, is the minimum $k$ 
such that $G$ is $k$-choosable.

Similarly to the chromatic number, one may seek for bounds on the choose number 
of a graph with few crossings or with independent crossings.
Thomassen's Five Colour Theorem~\cite{Tho94} states that if a graph has crossing number zero (i.e. is planar) then it is $5$-choosable.
A natural question is to ask whether the chromatic number is bounded in terms of its crossing number.
Erman et al.~\cite{EHLP11} observed that Thomassen's result can be extended to graphs with crossing
number at most $1$. Deleting one vertex per crossing yields $\ch(G)\leq 4 +cr(G)$.
Hence, what is the smallest integer $g(k)$ such that every graph $G$ with crossing number at most $k$ is $g(k)$-choosable?
Obviously, since $\chi(G)\leq \ch(G)$, we have $f(k)\leq g(k)$.

In this paper, we extend Erman et al. result in two ways.
We first show that every graph which can be made planar by the removal of an edge is $5$-choosable (Theorem~\ref{cr1choos}).
We then prove that is $g(2)=5$. In other words, every graph with crossing number $2$  is $5$-choosable\footnote{While writing this paper, we discovered that Dvo\v{r}\'ak et al.~\cite{DLR} independently proved this result.
Their proof has some similarity to ours but is different. They prove by induction a stronger result, while we
use the existence of a shortest path between the two crossings which satisfies some given properties.}. This generalizes the result of Oporowski and Zhao~\cite{OpZh} to list colouring.

\section{Planar graphs plus an edge} 

In order to prove its Five Colour Theorem, Thomassen~\cite{Tho94} showed  a stronger result.

\begin{definition}\rm
An {\it  inner triangulation} is a plane graph such that 
every face of $G$ is bounded by a triangle except its outer face which is bounded by a cycle.

Let $G$ be a plane graph and $x$ and $y$ two consecutive vertices on its outer face $F$.
A list assignment $L$ of $G$ is {\it $\{x,y\}$-suitable} if
\begin{itemize}
\item[-] $|L(x)|\geq 1$, $|L(y)|\geq 2$,
\item[-] for every $v\in V(F)\setminus \{x,y\}$, $|L(v)|\geq 3$, and
\item[-] for every $v\in V(G)\setminus V(F)$,   $|L(v)|\geq 5$.
\end{itemize}
\end{definition}

A list assignment of $G$ is {\it suitable} if it is $\{x,y\}$-suitable for some vertices $x$ and $y$ on the outer face of $G$.

The following theorem is a straightforward generalization of Thomassen's five colour Theorem which holds for non-separable plane graphs.

\begin{theorem}[Thomassen~\cite{Tho94}]\label{suitable}
If $L$ is a suitable list assignment of a plane graph $G$ then $G$ is
$L$-colourable.
\end{theorem}

This result is the cornerstone of the following proof.

\begin{theorem}\label{cr1choos}
Let $G$ be a graph.
If $G$ has an edge such that $G\setminus e$ is planar 
then $\ch(G)\leq 5$.
\end{theorem}
\begin{proof}
Let $e=uv$ be an edge of $G$ such that $G\setminus e$ is planar.
Let $G'$ be a planar triangulation containing $G\setminus e$ as a subgraph.
Without loss of generality, we may assume that
$u$ is on the outer triangle of $G'$.
The graph $G'-u$ has an outer cycle $C'$ whose vertices are the neighbours of $u$ in $G'$. 

Let $L$ be a $5$-list assignment of $G$.
Let $\alpha, \beta\in L(u)$.
Let $L'$ be the list-assignment of $G'-u$ defined by 
$L'(w)=L(w)\setminus \{\alpha, \beta\}$ if
$w\in V(C')$ and $L'(w)=L(w)$ otherwise.
Then $L'$ is suitable.
So $G'-u$ admits an $L'$-colouring by Theorem~\ref{suitable}.
This colouring may be extended into an $L$-colouring of $G$ by assigning to $u$ a colour in $\{\alpha, \beta\}$ different from the colour of $v$.

Hence $G$ is $5$-choosable.
\end{proof}

\section{Graphs with two crossings}

\subsection{Preliminaries}

We first recall the celebrated characterization of planar graphs due to Kuratowski~\cite{Kur30}.
See also \cite{Tho81} for a nice proof.

\begin{theorem}[Kuratowski~\cite{Kur30}]
A graph is planar if and only if it contains no minor  isomorphic to either $K_5$ or $K_{3,3}$.
\end{theorem}

Let $G$ be a plane graph and $x$, $y$ and $z$ three distinct vertices on the outer face $F$ of $G$.
A list assignment $L$ of $G$ is {\it $(x,y,z)$-correct} if
\begin{itemize}
\item[-] $|L(x)|=1 = |L(y)|$ and $L(x)\neq L(y)$,
\item[-] $|L(z)|\geq 3$,
\item[-] for every $v\in V(F)\setminus \{x,y,z\}$, $|L(v)|\geq 4$, and
\item[-] for every $v\in V(G)\setminus V(F)$,   $|L(v)|\geq 5$.
\end{itemize}

If $L$ is $(x,y,z)$-correct and $L(z)\geq 4$, we say that $L$ is {\it $\{x,y\}$-correct}.

\begin{lemma}\label{nice}
Let $G$ be an inner triangulation and $x$ and $y$ two distinct vertices on the outer face of $G$.
If $L$ is an $(x,y,z)$-correct list assignment of $G$ then $G$ is $L$-colourable.
\end{lemma}
\begin{proof}
We prove the result by induction on the number of vertices, the result holding trivially when $|V(G)|=3$.

Suppose first that $F$ has a chord $xt$. 
Then $xt$ lies in two unique cycles in $F\cup xt$, one $C_1$ containing 
$y$ and the other $C_2$.
For $i=1,2$, let $G_i$ denote the subgraph induced by the vertices
lying on $C_i$ or inside it. 
By the induction hypothesis, there exists an $L$-colouring $\phi_1$ of $G_1$.
Let $L_2$ be the list assignment on $G_2$ defined by
$L_2(t)=\{\phi_1(t)\}$ and
$L_2(u)=L(u)$ if $u\in V(G_2)\setminus \{t\}$. Let $z'=z$ if $z\in V(C_2)$ and $z'$ be any vertex of $V(C_2)\setminus \{x,t\}$ otherwise.
Then $L_2$ is $(x,t,z')$-correct for $G_2$ so
$G_2$ admits an $L_2$-colouring $\phi_2$ by induction hypothesis.
The union of $\phi_1$ and $\phi_2$ is an $L$-colouring of $G$.

Suppose now that $x$ has exactly two neighbours $u$ and $v$ on $F$.  Let $u,u_1,u_2
\dots, u_m, v$ be the neighbours of $x$ in their natural 
cyclic order around $x$. As $G$ is an  inner triangulation,
$uu_1u_2\cdots u_m,v=P$ is a path.
Hence the graph $G-x$ has $F'=P\cup (F-x)$ as outer face.

Assume first that $z\notin \{u,v\}$.
Then let $L'$ be the list assignment on $G-x$ defined by
$L'(w)=L(w)\setminus L(x)$ if $w\in N_G(x)$ and
$L'(w)=L(w)$ otherwise.
Clearly, $|L'(w)|\geq 3$ if $w\in F'$ and $|L'(w)|\geq 5$ otherwise. 
Hence, by Theorem~\ref{suitable}, $G-x$ admits an $L'$-colouring.
Colouring $x$ with the colour of its list, we obtain an $L$-colouring of $G$.

Assume now that $z\in \{u,v\}$, say $z=u$. 
Let $\alpha$ be a colour of $L(z)\setminus (L(x)\cup L(y))$.
Let $L'$ be the list assignment on $G-x$ defined by
$L'(z)=\{\alpha\}$,
$L'(w)=L(w)\setminus L(x)$ if $w\in N_G(x)\setminus \{z\}$ and
$L'(w)=L(w)$ otherwise.
Clearly, $L'$ is $(y,z,v)$-correct. 
Hence, by the induction hypothesis, $G-x$ admits an $L'$-colouring.
Colouring $x$ with the colour of its list, we obtain an $L$-colouring of $G$.
\end{proof}

\subsection{Nice, great and good paths}

Let $G$ be a graph and $H$ an induced subgraph of $G$.

We denote by $Z_H$ the set of vertices of $G$ which are adjacent to at least $3$ vertices of $H$. For every vertex
$v$ in $V(G)$, we denote by $N_H(v)$ the set of vertices of $H$ adjacent to $v$, and  we set $d_H(v)=|N_H(v)|$.

Let $L$ be a list assignment of $G$. For any $L$-colouring $\phi$ of $H$, we denote by $L_{\phi}$ the list assignment
of $G-H$ defined by $L_{\phi}(z)=L(z)\setminus \phi(N_H(z))$.
A vertex $z\in V(G-H)$ is {\it safe} (with respect to $\phi$), if $|L_{\phi}(z)|\geq 3$.
An $L$-colouring of $H$ is {\it safe} if  all vertices of $z\in V(G-H)$ are safe. 
Observe that if $L$ is a $5$-list assignment, then for any $L$-colouring $\phi$ of $H$, every vertex $z$ not in $Z_H$ has at most two neighbours in $H$ and therefore $|L_{\phi}(z)|\ge 3$. Hence $\phi$ is safe if and only if every vertex in $Z_H$ is safe.

Let $P = v_1 \cdots v_p$  be an induced path in $G$.  For $2\le i\le p-1 $, we denote by $[v_i]_P$, or simply $[v_i]$ if $P$ is clear from the context, the set $\{v_{i-1}, v_i, v_{i+1}\}$. We say that a vertex $z$ is adjacent to $[v_i]$ if it is adjacent to all vertices in the set $[v_i]$. Note that if $z$ is adjacent to $[v_i]$ then $z$ is not in $P$ as $P$ is induced.

\begin{lemma}\label{util}
Let  $P = v_1 \cdots v_p$  be an induced path in $G$, $x$ a vertex such that $N_P(x)=[v_{i+1}]$, $1\leq i\leq p-1$, and $\phi$ a colouring of $P-v_i$.
If $i=1$ or $\phi(v_{i-1}) =\phi(v_{i+1})$, then one can extend $\phi$ to $v_i$ such that $x$ is safe.
\end{lemma}
\begin{proof}
If $\{\phi(v_{i+1}), \phi(v_{i+2})\} \not\subset L(x)$, then assigning to $v_i$ any colour distinct from
$\phi(v_{i+1})$, we get a colouring of $P$ such that $x$ is safe.
So we may assume that $\{\phi(v_{i+1}), \phi(v_{i+2})\} \subset L(x)$.

If $\phi(v_{i+2})\in L(v_i)$, then setting $\phi(v_i)=\phi(v_{i+2})$, we have a colouring $\phi$ such that $x$ is safe.
If not, there is a colour $\alpha$ in $L(v_i)\setminus L(x)$. Necessarily, $\alpha\neq \phi(v_{i+1})$ and so one can colour $v_i$ with $\alpha$.
Doing so, we obtain a colouring such that $x$ is safe.
\end{proof}

Let $P = v_1 \cdots v_p$ be an induced path.
It is a {\it nice path} in $G$ if the following are true.
\begin{enumerate}
\item[(a)] for every $z\in Z_P$, $N_P(z) = [v_i]$ for some $2\le i\le p-1$;
\item[(b)] for every $2\le i\le p-1$, there are at most two vertices adjacent to $[v_i]$ and, if there are two such vertices, then the number of vertices adjacent to  $[v_{i-1}]$ or $[v_{i+1}]$ is at most $1$.
\end{enumerate}

It is a {\it great path} in $G$ if is is nice and satisfies the following extra property.
\begin{enumerate}
\item[(c)] for any $i<j$, if there are two vertices adjacent to $[v_i]$ and two vertices adjacent to $[v_j]$,  then the number of vertices adjacent to $[v_{i+1}]$ or  $[v_{j-1}]$ is at most $1$.
\end{enumerate}

A safe colouring of a path $P= v_1 \cdots v_p$ is {\it $\alpha$-safe} if $\phi(v_1) = \alpha$.

\begin{lemma}\label{nicepath}
If $P$ is a great path and $L$ is a $5$-list assignment of $G$, then for any $\alpha\in L(v_1)$, there exists an $\alpha$-safe $L$-colouring $\phi$ of $P$.
\end{lemma}
\begin{proof}

 We prove this result by induction on $p$, the number of vertices of $P$, the result holding trivially when $p\le 2$.
 
Assume now that $p\ge 3$. 
Since $P$ is great then every vertex of $Z_P$ adjacent to $v_1$ is also adjacent to $v_2$ and there are at most
two vertices of $Z_P$ adjacent to $[v_2]$.

Set $\phi(v_1)=\alpha$.

\begin{itemize}
\item[1.]
If there is no vertex adjacent to $[v_2]$, then by induction, for any  $\beta\in L(v_2)\setminus \{\alpha\}$, there is a $\beta$-safe $L$-colouring $\phi$ of $v_2 \cdots v_p$. Since $\phi(v_1)=\alpha$, $\phi$ is an $\alpha$-safe $L$-colouring of $P$.

\item[2.]
Assume now that there is a unique vertex $z$ adjacent to $[v_2]$.

If $\alpha\notin L(z)$, then by Case 1, there is an $\alpha$-safe $L$-colouring $\phi$ of $P$ in $G-z$.
It is also an  $\alpha$-safe $L$-colouring of $P$ in $G$ since $z$ is safe as $\alpha\notin L(z)$.
Hence we may assume that $\alpha\in L(z)$.

Assume there is a colour $\beta$ in  $L(v_2)\setminus\{\alpha\}$. By induction  there is a $\beta$-safe $L$-colouring $\phi$ of $v_2 \cdots v_p$. Since $\phi(v_1)=\alpha$, we obtain an $\alpha$-safe $L$-colouring of $P$ because  $z$ is safe as $\beta\notin L(z)$.
Hence we may assume that $L(v_2) = L(z)$. In particular, $\alpha\in L(v_2)$.
Let $\gamma$ be $\alpha$ if $\alpha \in L(v_3)$, and a colour in $L(v_3) \setminus L(v_2)$ otherwise.
We set $\phi(v_3)=\gamma$.
Observe that whatever colour is assigned to $v_2$, the vertex $z$ will be safe.
\begin{itemize}

\item[2.1.] Assume that no vertex is adjacent to $[v_3]$.
By induction hypothesis, there is a  $\gamma$-safe $L$-colouring $\phi$ of $v_3 \cdots v_p$.
Choosing $\phi(v_2)$ in $L(v_2)\setminus \{\alpha, \gamma\}$, we obtain 
an $\alpha$-safe $L$-colouring of $P$. 

\item[2.2.] Assume that exactly one vertex $t$ is adjacent to $[v_3]$.
By induction hypothesis, there is a  $\gamma$-safe $L$-colouring $\phi$ of $v_3 \cdots v_p$.
So far all the vertices except $t$ will be safe. So we just need to choose $\phi(v_2)$ so that $t$ is safe.

Observe that if $\{\gamma, \phi(v_4)\} \not \subset L(t)$, choosing any colour of $L(v_2)\setminus \{\alpha, \gamma\}$ will do the job.
So we may assume that $\{\gamma, \phi(v_4)\} \subset L(t)$.
If there is a colour $\beta\in L(v_2)\setminus (L(t) \cup \{\alpha\})$, then setting $L(v_2)=\beta$ will make $t$ safe.
So we may assume that $L(v_2)\setminus \{\alpha\}\subset L(t)$ and so $L(t)=L(v_2)\cup \{\gamma\}\setminus \{\alpha\}$.
Thus $\phi(v_4) \in L(v_2)\setminus \{\alpha, \gamma\}$.
Then setting $\phi(v_2)=\phi(v_4)$ makes $t$ safe.

\item[2.3.]
Assume that two vertices $t_1$ and $t_2$ are adjacent to $[v_3]$. Then no vertex is adjacent to $[v_4]$.
Therefore, it suffices to prove that there is an $\alpha$-safe $L$-colouring of $v_1v_2v_3v_4$.
Indeed, if we have such a colouring $\phi$, then by induction, $v_4\cdots v_p$ admits a $\phi(v_4)$-safe $L$-colouring $\phi'$.
The union of these two colourings is an $\alpha$-safe $L$-colouring of $P$.

If there exists $\beta\in L(v_4) \cap L(v_2)\setminus \{\alpha, \gamma\}$, then setting $\phi(v_2)=\phi(v_4)=\beta$, we obtain
an $\alpha$-safe $L$-colouring of $v_1v_2v_3v_4$.
Otherwise, $L(v_4)\setminus \{\gamma\}$ and $L(v_2)\setminus \{\alpha\}$ are disjoint.
Hence one can choose $\beta$ in $L(v_2)\setminus \{\alpha\}$  and  $\delta$ in $L(v_4)\setminus \{\gamma\}$ so that
$|\{\beta , \gamma, \delta\}\cap L(t_i)|\leq 2$ for $i=1,2$.
Setting $\phi(v_2)=\beta$ and $\phi(v_4)=\delta$, we obtain
an $\alpha$-safe $L$-colouring of $v_1v_2v_3v_4$.

\end{itemize}

\item[3.] Assume that two vertices $z_1$ and $z_2$ are adjacent to $[v_2]$. 

We claim that it suffices to prove that there is an $\alpha$-safe $L$-colouring of $v_1v_2v_3$.
 
 Let $j$ be the smallest index such that no vertex is adjacent to $[v_j]$. For the definition of $j$, consider there is no vertex adjacent to $[v_p]$ so that $j\le p$.
By the property (c) of great path, for all $3\leq i < j$, there is exactly one  vertex $z_i$ adjacent to $[v_i]$.
For $i=3, \ldots, j-1$, one after another, one can use Lemma~\ref{util} in the path $v_{i+1} \cdots v_1$ to extend $\phi$ to $v_{i+1}$, so that
$z_i$ is safe. Then applying induction on the path  $v_{j}\cdots v_p$, we obtain an $\alpha$-safe $L$-colouring.
This proves the claim.

Let us now prove that an  $\alpha$-safe $L$-colouring of $v_1v_2v_3$ exists.

If $\alpha\notin L(z_i)$, then any $\alpha$-safe $L$-colouring of $v_1v_2v_3$ in $G-z_i$ will be an $\alpha$-safe $L$-colouring in  $G$.  By Case 2, one can find such a colouring in $G-z_i$, so we may assume that
$\alpha\in L(z_i)$.

If there is a colour $\beta\in L(v_2)\setminus L(z_1)$, then set $\phi(v_2)=\beta$.  By Lemma~\ref{util} in the path $v_3 v_2 v_1$, one can choose $\phi(v_3)$ in $L(v_3)$ to obtain an $\alpha$-safe $L$-colouring of $v_1v_2v_3$.
Hence we may assume that $L(z_1)=L(v_2)$. Similarly, we may assume that $L(z_2)=L(v_2)$. Therefore, any $\alpha$-safe $L$-colouring of $v_1v_2v_3$ in $G-z_2$ will be an $\alpha$-safe $L$-colouring in  $G$. 
We can find such a colouring using Case 2.
\end{itemize}
\end{proof}

We say that an induced path $P = v_1 \cdots v_p$ is {\it good} path if either $P$ is great or $p\ge 4$ and there is a vertex $z\in Z_P$ adjacent to $v_1$ such that $\{v_1,v_4\} \subset N_P(z) \subseteq \{v_1, v_2, v_3, v_4\}$ satisfying the following conditions:
\begin{itemize}
\item $P$ is a great path in $G\setminus v_1z$. 
\item if two vertices distinct from $z$ are adjacent to $[v_2]$, then $N_P(z)=\{v_1, v_3, v_4\}$ and there is no vertex adjacent to $[v_3]$; and
\item if two vertices distinct from $z$ are adjacent to $[v_3]$, then $N_P(z)=\{v_1, v_2, v_4\}$ and there is no vertex adjacent to $[v_2]$.
\end{itemize}
Note that since $P$ is induced, then $z$ is not in $P$.

\begin{lemma}\label{goodpath}
If $P = v_1\cdots v_p$ is a good path and $L$ is a $5$-list assignment of $G$, then there exists a safe $L$-colouring of $P$.
\end{lemma}
\begin{proof}

If $P$ is great, then the result follows from Lemma~\ref{nicepath}. 
So we may assume that $P$ is not  great.
Let $z$ be the vertex of $Z_P$ such that $\{v_1,v_4\} \subset N_P(z) \subseteq \{v_1, v_2, v_3, v_4\}$.

If there is a colour $\alpha\in L(v_1)\setminus L(z)$, then let $\phi(v_1) = \alpha$ and use Lemma~\ref{nicepath} to colour $v_1 \cdots v_p$ in $G\setminus v_1z$. The obtained colouring $\phi$ is a safe $L$-colouring of $P$. For any $z'\in Z_{P}\setminus \{z\}$, we have  $|L_{\phi}(z')|\geq 3$ because 
$z'$ has the same neighbourhood in $G$ and $G\setminus v_1z$. 
Now $|L_{\phi}(z)|\geq 3$ since $\alpha \notin L(z)$, so $\phi$ is safe.
Henceforth, we assume that $L(v_1) = L(z)$.

\begin{itemize}
\item[1.] Assume first that $N_P(z) = \{v_1, v_2, v_3, v_4\}$. 

By the properties of a good path, at most one vertex $z'$ different from $z$ is adjacent to $[v_2]$.

\begin{itemize}
\item[1.1.]  Assume first that $z$ is the unique vertex adjacent to $[v_3]$.

If there is a colour $\alpha\in L(z)\cap L(v_3)$, then set $\phi(v_1)=\phi(v_3)=\alpha$.
By Lemma~\ref{nicepath}, one can extend $\phi$ to $v_3 \cdots v_p$ so that all vertices of $Z_P$ but $z$ are safe.
Then by Lemma~\ref{util} applied to $v_2\cdots v_p$, one can choose $\phi(v_2)\in L(v_2)$ so that $z$ is safe for $P-v_1$.
Since $\phi(v_1)=\phi(v_3)$, then $\phi$ is a proper colouring and $z$ is safe for $P$.
Hence $\phi$ is  a safe $L$-colouring of $P$.
So we may assume that $L(z)\cap L(v_3)=\emptyset$.

If there exists $\beta\in L(v_2)\setminus L(z)$, then set $\phi(v_2)=\beta$.
By Lemma~\ref{nicepath}, one can extend $\phi$ to $v_2 \cdots v_p$ so that all vertices of $Z_P$ but $z$ and $z'$ are safe.
Observe that necessarily $z$ will be safe because $\phi(v_2)\notin L(z)$ and  $\phi(v_3)\notin L(z)$.
By Lemma~\ref{util}, one can extend $\phi$ to $v_1$ so that $z'$ is safe, thus getting a safe $L$-colouring of $P$.
So we may assume that $L(v_2)=L(z)$.

We have $|L(v_2)\cup L(v_3)|=10 \geq |L(z')|$. So we can find $\alpha\in L(v_2)$ and $\beta\in L(v_3)$ so that
$|\{\alpha, \beta\} \cap L(z')|\leq 1$. 
Using Lemma~\ref{nicepath} take a $\beta$-safe $L$-colouring $\phi$ of the path $v_3v_4 \dots v_p$ and set
$\phi(v_2)=\alpha$.
If $\phi(v_4)\in L(z)\setminus \{\alpha\}$, then colour $v_1$ with $\phi(v_4)$, otherwise colour it with any colour distinct from $\alpha$.
This gives a safe $L$-colouring of $P$.

\item[1.2] Assume now that a vertex $y\neq z$ is adjacent to $[v_3]$.

\begin{itemize}
\item Suppose that a vertex $t$ is adjacent to $[v_4]$. Then $z'$ does not exist.

If there is a colour $\alpha \in L(v_2)\setminus L(z)$, then using Lemma~\ref{nicepath} take an $\alpha$-safe $L$-colouring $\phi$ of  $v_2 \cdots v_p$.  If $\phi(v_3)\notin L(z)$, then $z$ would be safe whatever colour we assign to $v_1$, so
there is a safe $L$-colouring of $P$.
If  If $\phi(v_3)\in L(z)$, then setting $\phi(v_1)=\phi(v_3)$, we obtain a safe $L$-colouring of $P$.
So we may assume that $L(v_2)=L(z)$.

If there is a colour $\alpha$ in $L(z)\cap L(v_4)$, then set $\phi(v_2)=\phi(v_4)=\alpha$.
Then $y$ will be safe. Extend $\phi$ to $v_4\cdots v_p$ by Lemma~\ref{nicepath}.
Then all the vertices are safe except $t$ and $z$.
By Lemma~\ref{util}, one can choose $\phi(v_3)$ so that $t$ is safe.
If $\phi(v_3)\in L(z)$, then setting $\phi(v_1)=\phi(v_3)$, we get a safe $L$-colouring of $P$.
If  $\phi(v_3)\notin L(z)$, then whatever colour we assign to $v_1$, we obtain a safe colouring of $P$.
Hence we may assume that $L(z)\cap L(v_4)=\emptyset$. 
By Lemma~\ref{nicepath}, there is a safe $L$-colouring of $P$ in $G\setminus zv_4$. This colouring is also a safe colouring
of $P$ in $G$, since $\phi(v_4)$ is not in $L(z)$.

\item If no vertex is adjacent to $[v_4]$, then $z'$ may exist.
In this case, it is sufficient to prove that there exists a safe $L$-colouring of $v_1v_2v_3v_4$.
Indeed, if there is such a colouring $\phi$, then by Lemma~\ref{nicepath}, it can be extended to a safe $L$-colouring of $P$.

Symmetrically to the way we proved the result when $L(v_1)\neq L(z)$, one can prove it when  $L(v_4)\neq L(z)$.
Hence we may assume that $L(v_4)=L(z)$.

Assume that there is a colour $\alpha \in L(v_2)\cap L(z)$.  Set $\phi(v_2)=\phi(v_4)=\alpha$.
If there is a colour $\beta \in L(v_3)\setminus L(z)$, then set $\phi(v_3)=\beta$ so that $z$ will be safe
and extend $\phi$ with Lemma~\ref{util} so that $z'$ is safe to obtain a safe colouring
of $v_1v_2v_3v_4$ in $G$.
If  $L(v_3) = L(z)$, then assign to $v_1$ and $v_3$ a same colour in $L(z)\setminus \{\alpha\}$ to get a safe colouring
of $v_1v_2v_3v_4$.

Hence we may assume that $L(v_2)\cap L(z)=\emptyset$. Symmetrically, we may assume that $L(v_3)\cap L(z)=\emptyset$.
By Lemma~\ref{nicepath}, there exists a safe colouring $\phi$ of $v_1v_2v_3v_4$ in $G-z$.
It is also a safe colouring of $v_1v_2v_3v_4$ in $G$ because $\phi(v_2)$ and $\phi(v_3)$ cannot be in $L(z)$.

\end{itemize}

\end{itemize}

\item[2.] Assume now that $N_P(z)= \{v_1,v_3,v_4\}$.

If no vertex is adjacent to $[v_2]$, then using Lemma~\ref{nicepath} take a safe $L$-colouring of $v_2 \dots v_p$.
If $\phi(v_3)\in L(z)$, then set $\phi(v_1)=\phi(v_3)$. If not colour $v_3$ with any colour in $L(z)\setminus \{\phi(v_2)\}$. This gives a safe $L$-colouring of $P$.
Hence we may assume that a vertex $t$ is adjacent to $[v_2]$.

By the properties of a good path, we know that at most one vertex, say $u$, is adjacent to $v_3$.
If $L(v_3)\cap L(z)$ is empty, then any safe $L$-colouring of $P$ given by Lemma~\ref{nicepath} in $G\setminus zv_1$ would be a safe $L$-colouring of $P$.
Hence we may assume that there is a colour $\alpha$ in $L(v_3)\cap L(z)$.
Set $\phi(v_1)=\phi(v_3)=\alpha$ and apply Lemma~\ref{nicepath} to $v_3 \dots v_p$.
Then by Lemma~\ref{util}, we can choose $\phi(v_2)$ so that the possible vertex $u$ is safe.
This gives a safe colouring of $P$.

\item[3.] Assume that $N_P(z)= \{v_1,v_2,v_4\}$.

Suppose no vertex is adjacent to $[v_2]$. 
By Lemma~\ref{nicepath}, there is a safe$L$- colouring of $v_2 \dots v_p$. 
Set $\phi(v_1)=\phi(v_4)$ if $\phi(v_4)\in L(z) \setminus \{\phi(v_2)\}$, and let $\phi(v_1)$ be any colour of $L(v_1)\setminus \{\phi(v_2)\}$ otherwise.
Doing so $z$ is safe and so $\phi$ is a safe $L$-colouring of $P$.
Hence we may assume that a vertex $u$ is adjacent to $[v_2]$.
By definition of good path, it is the unique vertex adjacent to $[v_2]$.

Suppose that there exists a colour $\beta$ in $L(v_2)\setminus L(z)$.
By Lemma~\ref{nicepath}, there is a safe colouring $\phi$ of $v_2 \dots v_p$ such that $\phi(v_2)=\beta$.
By Lemma~6, it can be extended to $v_1$ so that $u$ is safe. This yields a safe $L$-colouring of $P$.
Hence we may assume that $L(v_2)=L(z)$.

If $L(v_4)\cap L(z)=\emptyset$, then 
in every colouring of $P$, the vertex $z$ will be safe.
Hence any safe colouring of $P$ in $G-z$, (there is one by Lemma~\ref{nicepath}) is a safe $L$-colouring of $P$ in $G$.
So we may assume that there exists a colour $\alpha\in L(v_4)\cap L(z)$.

Assume that at most one vertex $s$ is adjacent to $[v_4]$. Set $\phi(v_2)=\phi(v_4)=\alpha$ so that $z$ and all the vertices adjacent to $[v_3]$ will be safe.
By Lemma~\ref{nicepath}, there is an $\alpha$-safe colouring of $v_4\dots v_p$. 
Now by Lemma~\ref{util}, one can extend $\phi$ to $v_3$ so that $s$ (if it exists) is safe, and then again by Lemma~\ref{util} extend it to $v_1$ so that $u$ is safe.
This gives a safe $L$-colouring of $P$.
 So we may assume that two vertices $s$ and $s'$ are adjacent to $[v_4]$.

Assume that there is a vertex $t$ adjacent to $[v_3]$, then there is no vertex adjacent to $[v_5]$.
Hence it suffices to find a safe $L$-colouring of $v_1v_2v_3v_4v_5$.
Indeed, if we have such a colouring $\phi$, then using Lemma~\ref{nicepath}, one can extend it to
a safe $L$-colouring of $P$.
Set $\phi(v_2)=\phi(v_4)=\alpha$. Doing so $t$ and $z$ will be safe.
If $\alpha$ or some colour $\beta \in L(v_5)\setminus \{\alpha\}$ is not contained in one of lists $L(s)$ and $L(s')$, say $L(s')$. Then colouring $v_5$ with $\beta$, if it exists, or any other colour otherwise, the vertex $s'$ will also be safe.
By Lemma~\ref{util}, one can colour $v_3$ so that $s$ is safe. By Lemma~\ref{util}, one can then
colour $v_1$ to obtain a colouring for which $u$ is safe. This  $L$-colouring of $v_1v_2v_3v_4v_5$ is safe.
Hence, we may assume that $L(s)=L(s')=L(v_5)$.
Colour $v_5$ with any colour in $L(v_5)\setminus \{\alpha\}$.
Using Lemma~\ref{util}, colour $v_3$ so that $s$ is safe. Then $s'$ will be also safe because $L(s)=L(s')$. Again by Lemma~\ref{util}, colour $v_1$ so that $u$ is safe to obtain a safe colouring of $v_1v_2v_3v_4v_5$.

Assume finally that no vertex is adjacent to $[v_3]$.
By Lemma~\ref{nicepath}, there is a safe $L$-colouring $\phi$ of $v_3 \dots v_p$.
If $\phi(v_4)\notin L(z)$, then assign to $v_2$ any colour in $L(v_2)\setminus \{\phi(v_3)\}$.
If not, then set $\phi(v_2)=\phi(v_4)$. (This is possible since $L(v_2)=L(z)$.) Then $z$ will be safe.
 By Lemma~\ref{util}, colour $v_1$ so that $u$ is safe to obtain a safe $L$-colouring of $P$.
\end{itemize}
\end{proof}

\subsection{Main theorem}

A drawing of $G$ is {\it nice} if two edges intersect at most once.
It is well known that every graph with crossing number $k$ has a nice drawing with at most $k$ crossings.
(See~\cite{EHLP11} for example.) In this paper, we will only consider nice drawings.
 Thus a crossing is uniquely defined
by the pair of edges it belongs to. Henceforth, we will confound a crossing with this set of two edges.
The {\it cluster} of a crossing $C$ is the set of endvertices of its two edges and is denoted $V(C)$.

\begin{theorem}\label{2dep}
Let $G$ be a graph having a drawing in the plane with two crossings. Then $\ch(G)\leq 5$.
\end{theorem}

\begin{proof}
By considering a counter-example $G$ with the minimum number of vertices. Let $L$ be a $5$-list assignment of $G$ such that $G$ is not $L$-colourable.

Let $C_1$ and $C_2$ be the two crossings.
By Theorem~\ref{cr1choos}, $C_1$ and $C_2$ have no edge in common.
Set $C_i=\{v_iw_i, t_iu_i\}$.
Free to add edges and to redraw them along the crossing, we may assume that $v_iu_i$, $u_iw_i$, $w_it_i$ and $t_iv_i$ are edges and that the $4$-cycle $v_iu_iw_it_i$ has no vertex inside but the two edges of $C_i$.
In addition, we assume that $u_1v_1t_1w_1$ appear in clockwise order around the crossing point of $C_1$ and that
$u_2v_2t_2w_2$ appear in counter-clockwise order around the crossing point of $C_2$. 
Free to add edges, we may also assume that $G\setminus \{v_1w_1, v_2w_2\}$ is a triangulation of the plane. In the rest of the proof,
for convenience, we will refer to this fact by writing that $G$ is {\it triangulated}.

\begin{claim}\label{deg5}
Every vertex of $G$ has degree at least $5$.
\end{claim}
\begin{proof}
Suppose not. Then $G$ has a vertex $x$ of degree at most $4$.
By minimality of $G$, $G-x$ has an $L$-colouring $\phi$.
Now assigning to $x$ a colour in $L(x)\setminus \phi(N(x))$ we obtain an $L$-colouring of $G$, a contradiction.
\end{proof}

A cycle is {\it separating} if none of its edges is crossed and
both its interior and exterior contain at least one vertex.
A cycle is {\it nicely separating} if
it is separating and its interior or its exterior has no crossing.

\begin{claim}\label{sep3}
$G$ has no nicely separating triangle.
\end{claim}
\begin{proof}
Assume, by way of contradiction, that a triangle $T=x_1x_2x_3$ is nicely separating. 
Let $G_1$ (resp. $G_2$) be the subgraph of $G$ induced by the vertices on $T$ or outside $T$ (resp. inside $T$).
Without loss of generality, we may assume that $G_2$ is a plane graph.

By minimality of $G$, $G_1$ has an $L$-colouring $\phi_1$.
Let $L_2$ be the list assignment of $G_2$ defined by
$L_2(x_1)=\{\phi_1(x_1)\}$, $L_2(x_2)=\{\phi_1(x_1), \phi_1(x_2)\}$,
$L_2(x_3)=\{\phi_1(x_1), \phi_1(x_2), \phi_1(x_3)\}$, and
$L_2(x)=L(x)$ for every vertex inside $T$.
Then $L_2$ is a suitable list assignment of $G_2$, so by Theorem~\ref{suitable}, $G_2$ admits an $L_2$-colouring $\phi_2$.
Observe that necessarily $\phi_2(x_i)=\phi_1(x_i)$.
Hence the union of $\phi_1$ and $\phi_2$ is an $L$-colouring of $G$, a contradiction.
\end{proof}

\begin{claim}\label{4cycle}
Let $C=abcd$ be a $4$-cycle with no crossing inside it.
If $a$ and $c$ have no common neighbour inside $C$ then $C$ has no vertex in its interior.
\end{claim}
\begin{proof}
Assume by way of contradiction that the set $S$ of vertices inside $C$ is not empty.

Then $ac$ is not an edge otherwise one of the triangles $abc$ and $acd$ would be nicely separating.
Since $G$ is triangulated, the neighbours of $a$ (resp. $c$) inside $C$ plus $b$ and $d$ (in cyclic order around $a$ (resp. $c$)) form
a $(b,d)$-path $P_a$ (resp. $P_c$).
The paths $P_a$ and $P_c$ are internally disjoint because $a$ and $c$ have no common neighbour inside $C$. Hence $P_a\cup P_c$ is a cycle
$C'$. Furthermore $C'$ is the outerface of $G'=G\langle S\cup\{b,d\}\rangle$.

By minimality of $G$, $G_1=(G-S)\cup bd$ admits an $L$-colouring $\phi$.
Let $L'$ be the list-colouring of $G'$ defined by
$L'(b)=\{\phi(b)\}$, $L'(d)=\{\phi(d)\}$,
$L'(x)=L(x)\setminus \{\phi(a)\}$ if $x$ is an internal vertex of $P_a$,
$L'(x)=L(x)\setminus \{\phi(c)\}$ if $x$ is an internal vertex of $P_c$,
and $L'(x)=L(x)$ if $x\in V(G'-C')$.
Then $L'$ is a $\{b,d\}$-correct list assignment of $G'$.
Hence, by Lemma~\ref{nice}, $G'$ admits an $L'$-colouring $\phi'$.
The union of $\phi$ and $\phi'$ is an $L$-colouring of $G$, a contradiction.
\end{proof}

\begin{claim}\label{sep4}
$G$ has no nicely separating $4$-cycle.
\end{claim}
\begin{proof}
Suppose not. Then there exists a nicely separating $4$-cycle
$abcd$.
Let $b=z_1, z_2, \dots , z_{p+1}=d$ be the common neighbours of $a$ and $c$ in clockwise order around $a$. 
By Claim~\ref{4cycle}, we have $p\geq 2$.
Each of the $4$-cycles
$az_icz_{i+1}$, $1\leq i\leq p$ has empty interior by Claim~\ref{4cycle}. So $z_2$ has degree at most $4$. This contradicts Claim~\ref{deg5}.
\end{proof}

A path $P$ is {\it friendly} if there are two adjacent vertices $x$ and $y$ such that
$|N_P(x)|\leq 4$, $|N_P(y)|\leq 3$ and $P$ is good in $G-\{x,y\}$.
A path $P$ {\it meets} a crossing if it contains at least one endvertex of each of the two crossed edges.
A {\it magic path} is a friendly path meeting both crossings.

\begin{claim}\label{nomagic}
$G$ has no magic path $Q$. 
\end{claim}
\begin{proof}
Suppose for a contradiction that $G$ has a magic path $Q$.
Then there exists two adjacent vertices $x$ and $y$ such that $|N_Q(x)|\leq 4$, $|N_Q(y)|\leq 3$ and $P$ is good in $G-\{x,y\}$. Lemma~\ref{goodpath}, there in a $L$-colouring $\phi$ of $Q$ such that every vertex $z$ of $(G-Q)-\{x,y\}$
satisfies $|L_{\phi}(z)|\geq 3$. Now  $|L_{\phi}(x)|\geq 1$ and $|L_{\phi}(y)|\geq 2$, because  $|N_Q(x)|\leq 4$ and $N_Q(y)\leq 3$ 
Since $Q$ meets the two crossings, $G-Q$ is planar. Furthermore, 
$G-Q$ may be drawn in the plane such that all the vertices on the outer face are those of $N(Q)$.
Hence $L_{\phi}$ is a suitable assignment of $G-Q$.
Hence by Theorem~\ref{suitable}, $G-Q$ is $L_{\phi}$-colourable and so $G$ is $L$-colourable, a contradiction.
\end{proof}

In the remaining of the proof, we shall prove that $G$ contains a magic path, thus getting a contradiction.
Therefore, we consider {\it shortest $(C_1,C_2)$-paths}, that are paths joining $C_1$ and $C_2$ with the smallest number of edges.
We first consider the cases when the distance between $C_1$ and $C_2$ is 0 or 1.
We then deal with the general case when $dist(C_1,C_2)\geq 2$.

\begin{claim}\label{dist0}
$dist(C_1,C_2) >0$.
\end{claim}
\begin{proof}
Assume for a contradiction that $dist(C_1,C_2) =0$. Then, without loss of generality, $v_1 = v_2$. 
Note that $u_1\neq u_2$ as otherwise the path $u_1v_1$ would be magic, contradicting Claim~\ref{nomagic}.
Similarly, we have $t_1\neq t_2$.


Note that $w_1$ is not adjacent to $u_2$ for otherwise both the interior and exterior of $w_1 u_1 v_1 u_2$ would contain at least one neighbour of $u_1$ by Claim~\ref{deg5}. 
Thus this $4$-cycle would be nicely separating, a contradiction to Claim~\ref{sep4}.
Henceforth, by symmetry, $w_1$ is not adjacent to $u_2$ nor $t_2$ and $w_2$ is not adjacent to $u_1$ nor $t_1$.


If $u_1$ is not adjacent to $u_2$, then consider the induced path $Q = u_1 v_1 u_2$.
Since $w_1$ and $w_2$ are not adjacent to $u_2$ and $u_1$, respectively, then $\{w_1, w_2\}\cap Z_Q = \emptyset$.
The vertices $t_1$ and $t_2$cannot be both in $Z_Q$ for otherwise $u_1t_2$ and $u_2t_1$ would cross.
Furthermore, if $z_1$ and $z_2$ are distinct vertices in $Z_Q\setminus \{t_1,t_2\}$, then either $u_1v_1u_2z_1$ nicely separates $z_2$ or $u_1v_1u_2z_2$ nicely separates $z_1$ contradicting Claim~\ref{sep4}. 	
Thus, $|Z_Q|\le 2$ and $Q$ is magic contradicting Claim~\ref{nomagic}.
Henceforth, $u_1$ is adjacent to $u_2$,  and, by a symmetrical argument, $t_1$ is adjacent to $t_2$.


If $u_1$ is adjacent to $t_2$, then both the interior and exterior of $u_1u_2w_2t_2$ contain at least one neighbour of $w_2$ by Claim~\ref{deg5}. Thus this $4$-cycle would be nicely separating, a contradiction to Claim~\ref{sep4}.
Henceforth,  $u_1$ is not adjacent to $t_2$, and symmetrically $t_1$ is not adjacent to $u_2$.


Therefore $Q = u_1v_1t_2$ is an induced path.
Note that $Z_Q\subseteq N(v_1)$.
The triangles $v_1u_1u_2$ and $v_1t_1t_2$ together with Claim~\ref{sep3} imply that $N(v_1) = \{u_1, u_2, t_1, t_2, w_1, w_2\}$.
Since $w_1$ is not adjacent to $t_2$ and $w_2$ is not adjacent to $u_1$, then $Z_Q = \{u_2, t_1\}$.
Thus $Q$ is magic contradicting Claim~\ref{nomagic}.
\end{proof}

\begin{claim}\label{nodiag}
Let $i\in \{1,2\}$ and $x$ a vertex not in $C_i$. 
Then at most one vertex in $\{u_i, t_i\}$ is adjacent to $x$ and at most one vertex in $\{v_i, w_i\}$ is adjacent to $x$.

\end{claim}
\begin{proof}
Assume for a contradiction that $x$ is adjacent to both $u_i$ and $t_i$.
Observe that the edges $u_ix$ and $t_ix$ are not crossed since $dist(C_1,C_2)\geq 1$.
Then one of the two $4$-cycles $u_iv_it_ix$  and $u_iw_it_ix$ is nicely separating. Thus the region bounded by this cycle has no vertex by Claim~\ref{sep4}. Hence either $d(v_i)\le 4$ or $d(w_i)\le 4$. This  contradicts Claim~\ref{deg5}.

Similarly, one shows that  at most one vertex in $\{v_i, w_i\}$ is adjacent to $x$.
\end{proof}

\begin{claim}\label{dist1}
$dist(C_1,C_2) >1$.
\end{claim}
\begin{proof}
Assume for a contradiction that $dist(C_1,C_2)=1$.
Without loss of generality, we may assume that $v_1v_2\in E(G)$.

\vspace{6pt}

Let us first show that without loss of generality, we may assume that $u_1$ is not adjacent to $v_2$ and $u_2$ is not adjacent to $v_1$.
By symmetry, if  $t_1$ is not adjacent to $v_2$ and $t_2$ is not adjacent to $v_1$, then we get the result by renaming
swapping the names of $u_i$ and $t_i$, $i=1,2$.
Thus by symmetry and by Claim~\ref{nodiag}, if it not the case, then $u_1v_2 \in E(G)$ and $v_1t_2\in E(G)$.
Moreover $w_1v_2$ is not an edge by Claim~\ref{nodiag}.
Hence renaming $u_1$, $v_1$, $t_1$, $w_1$ into  $v_1$, $t_1$, $w_1$, $u_1$ respectively, we are in the desired configuration.
\vspace{6pt}

The vertices $u_1$ and $u_2$ are not adjacent, for otherwise the cycle $u_1v_1v_2u_2$ would be nicely separating since $G$ is triangulated and $u_1v_2$ and $u_2v_1$ are not edges. So $Q$ is an induced path.

A vertex of $Z_Q$ is {\it goofy} if it is adjacent to $u_1$ and $u_2$.

\begin{itemize}

\item Suppose first that there is a goofy vertex $z'$ not in $C_1\cup C_2$.

Without loss of generality, we may assume that $z'$ is adjacent to $u_1$, $v_1$ and $u_2$.
If the crossing $C_1$ is inside $z'u_1v_1$, then consider the path $R=t_1v_1v_2u_2$.
It is induced since  $z'u_1v_1$ separates $t_1$ from $v_2$ and $u_2$.
Moreover all the neighbours of $t_1$ are inside $z'u_1v_1$, so they have at most two neighbours in $R$ except for $u_1$ which is not adjacent to $v_2$ nor to $u_2$.
Hence the vertices of $Z_R$ are all adjacent to $\{v_1,v_2,u_2\}$. Moreover $w_2\notin Z_R$ because $w_2v_1$ is not an edge by Claim~\ref{nodiag}.
Hence by planarity of $G-\{w_1,w_2\}$, there are at most two vertices adjacent to $\{v_1,v_2,u_2\}$.
Thus $R$ is magic, a contradiction.

Hence we may assume that $C_1$ is outside $z'u_1v_1$. The $4$-cycle $z'v_1v_2u_2$ is not nicely separating by Claim~\ref{sep4}, and $G$ is triangulated. So  $z'v_2\in E(G)$ because $v_1$ is not adjacent to $u_2$.  So $z'$ is adjacent to all vertices of $Q$.

Then there is no other vertex $z''$ in $Z_Q\setminus \{C_1\cup C_2\}$, for otherwise one of the crossing $C_i$ is inside $u_iv_iz''$ and
as above, we obtain the contradiction that $R$ is magic.

Now $w_1u_2$ is not an edge, for otherwise $w_1u_1z'u_2$ would be separating since $d(u_1)\geq 5$, a contradiction to Claim~\ref{sep4}.
Similarly, $w_2u_1$ is not an edge.
 Hence $Z_Q\subset \{z',t_1,t_2\}$.
Now one of the edges $t_1u_2$ and $t_2u_1$ is not in $E(G)$, since otherwise they would cross.
Without loss of generality, $t_1$ is not adjacent to $u_2$.
Then $Q$ is good in $G-t_2$, and so $Q$ is magic. This contradicts Claim~\ref{nomagic}.

\item Suppose now that all the goofy vertices of $Z_Q$ are in $C_1\cup C_2$.

Suppose first that $w_1$ is in $Z_Q$, then $w_1u_2$ is an edge 
because  $w_1$ is not adjacent to $v_2$ according to Claim~\ref{nodiag}. 
Thus $t_2$ and $w_2$ are not adjacent to $u_1$. So $w_2\notin Z_Q$ and
$N_Q(t_2)\subset \{v_1,v_2,u_2\}$, so $t_2$ is not goofy.
Moreover by planarity of $G-\{w_1,w_2\}$, there is at most two vertices adjacent $\{v_1,v_2,u_2\}$.
Furthermore, all the vertices distinct from $t_1$ and adjacent to $\{u_1,v_1, v_2\}$ are in the region bounded by $w_1v_1v_2u_2$ containing $u_1$. Therefore there is at most one such vertex.
Hence $Q$ is good in $G-\{w_1,t_1\}$.  Thus $Q$ is magic and contradicts Claim~\ref{nomagic}.

Similarly, we get a contradiction if $w_2\in Z_Q$. 
So $Z_Q\cap (C_1\cup C_2)\subseteq \{t_1,t_2\}$.
Then easily $Q$ is good in $G-t_2$ and so $Q$ is magic. This contradicts Claim~\ref{nomagic}.
\end{itemize}%
\end{proof}

\begin{claim}\label{shortestnicepath}
Some of the shortest $(C_1,C_2)$-paths is nice.
\end{claim}
\begin{proof}
Let $P = x_1x_2\cdots x_p$ be any  shortest $(C_1,C_2)$-path. Then no vertex in $C_1$ is adjacent to a vertex in $P-\{x_1,x_2\}$. Therefore, $V(C_1)\cap Z_P = \emptyset$. Similarly, we have $V(C_2)\cap Z_P = \emptyset$. Hence the graph $G'$ induced by  $V(P)\cup Z_P$ is planar as it contains exactly one vertex from each crossing.

Any vertex not in $P$ can be adjacent only to vertices of $P$ at distance at most two from each other, otherwise
there would be a $(C_1,C_2)$-path shorter than $P$. Thus, if $z\in Z_P$, then $z$ has precisely three neighbours in $P$. Moreover, there exists an $i\in \{2, \dots, p-1\}$ such that $N_P(z) = [x_i]$.

If there are distinct vertices $z_1, z_2, z_3 \in Z_P$ such that $N_P(z_1) = N_P(z_2) = N_P(z_2) = [x_i]$ for some value of $i$, then the subgraph of $G'$ induced by $\{z_1, z_2, z_3\}\cup \{x_{i-1}, x_i, x_{i+1}\}$ contains a $K_{3,3}$. By Kuratowski's Theorem, this contradicts the fact that $G'$ is planar.  Therefore, for every $2\leq i \leq p-1$, there are at most two vertices in $Z_P$ adjacent to $[x_i]$.

\vspace{6pt}

Let  $z_1, z_2 \in Z_P$ be such that $N_P(z_1) = N_P(z_2) = [x_i]$. The edges of $H = G[\{z_1, z_2\}\cup [x_i]]$ separate the plane into five regions $R_1, \dots, R_5$ as follows. Let $R_1$ be the region bounded by $x_{i-1}x_iz_1$ not containing the vertex $z_2$, $R_2$ be the region bounded by $x_ix_{i+1}z_1$ not containing the vertex $z_2$, $R_3$ be the region bounded by $x_{i-1}x_iz_2$ not containing the vertex $z_1$, $R_4$  be the region bounded by $x_ix_{i+1}z_2$ not containing the vertex $z_1$ and $R_5$ be the region bounded by $x_{i-1}z_1x_{i+1}z_2$ not containing $x_i$ (see Figure~\ref{sameneigh}). Since $(V(C_1)\cup V(C_2))\cap Z_P = \emptyset$ and $P$ is a shortest $(C_1,C_2)$-path, then no edge in $H$ is crossed.
\begin{figure}[!hbt]
\begin{center}
\scalebox{.7}{\input{sameneigh.pspdftex}}
\caption{Regions $R_1$, $R_2$, $R_3$, $R_4$ and $R_5$.}\label{sameneigh}
 \end{center}
\end{figure}
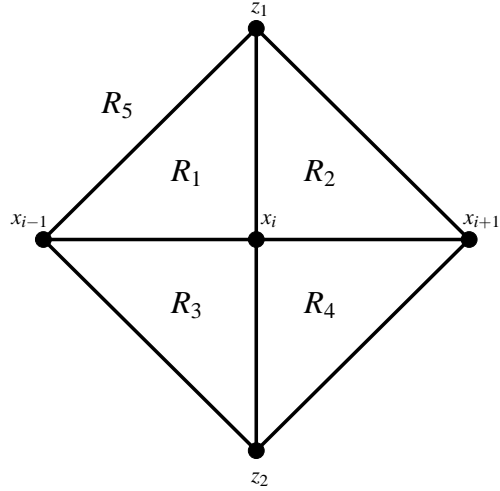

\vspace{6pt}

Let $J_P$ be the subset of $\{3, \ldots, p-2\}$ such that for $j\in J_P$, there are two vertices in $Z_P$ adjacent to $[x_j]$ and at least one vertex adjacent to $[x_{j-1}]$ and another adjacent to $[x_{j+1}]$.
The path $P$ is said to be {\it semi-nice} if $J_P=\emptyset$.

Let us first prove that some of the  shortest $(C_1,C_2)$-paths is semi-nice.

\begin{itemize}
\item[] Suppose for a contradiction that no shortest $(C_1,C_2)$-path is semi-nice.
Let $P$ be a shortest $(C_1,C_2)$-path that maximizes the smallest index $i$ in $J_P$. 
Let  $z_1, z_2 \in Z_P$ be such that $N_P(z_1) = N_P(z_2) = [x_i]$.
 
Let $z\in Z_P$ be a vertex adjacent to $[x_{i+1}]$. If $C_2$ is in $R_5$, then so is $x_{i+2}$ and we get a contradiction from the fact that either $zx_i$ or $zx_{i+2}$ must cross an edge of $H$. Since $P$ defines a path between $x_{i+1}$ and $V(C_2)$, then $C_2$ must be either in $R_2$ or in $R_4$ (say $R_4$). Similarly, $C_1$ is either in $R_1$ or in $R_3$. The cycle $x_{i-1}x_ix_{i+1}z_2$ is not be a nicely separating cycle by Claim~\ref{sep4}, so $C_1$ must be in $R_1$. Now, by Claim~\ref{sep3}, $R_2$ and $R_3$ are empty, and, by Claim~\ref{sep4},  there is no vertex in $R_5$. Since $P$ is a shortest path, $x_{i-1}x_{i+1}$ is not an edge and therefore $z_1$ is adjacent to $z_2$ as $G$ is triangulated.

Now, consider the path $P'$ obtained from $P$ by replacing $x_i$ with $x'_i = z_2$. Note that $P'$ is also a shortest path and that both $z_1$ and $x_i$ are adjacent to $[x'_i]$. Since no edge in $H$ is crossed, for any $v\in V(G)\setminus (\{z_1, z_2\}\cup [x_i])$,  if $v$ is adjacent to $x_{i-1}$ then it must be in $R_1$ and if $v$ is adjacent to  $z_2$ then it must be in $R_4$. Therefore, there is no vertex in $Z_{P'}$ adjacent to $\{x_{i-2}, x_{i-1},  z_2\}$. This implies that if $j\in J_{P'}$, then either $j \le i-3$ or $j \ge i+1$. Note that if $j\in J_{P'}$ and  $j \le i-3$, then $j\in J_P$. As $i$ is the minimum of $J_P$, the minimum of $J_{P'}$ is at least $i+1$. This contradicts our choice of $P$.
\end{itemize}
\vspace{6pt}

Let $K_P$ be the subset of $\{2, \ldots, p-1\}$ such that for $k\in K_P$, there are two vertices in $Z_P$ adjacent to $[x_k]$ and two vertices adjacent to $[x_{k+1}]$. 
Observe that a nice path $P$ is a semi-nice path such that $K_P$ is empty, that is a path such that $J_P$ and $K_P$ are empty. 

Suppose, by way of contradiction, that every $(C_1,C_2)$-shortest path is not nice.
Then consider the semi-nice  $(C_1,C_2)$-shortest path that maximizes the minimum of $K_P$.

Let  $z_1, z_2, z_3, z_4 \in Z_P$ be such that $N_P(z_1) = N_P(z_2) = [x_i]$ and $N_P(z_3) = N_P(z_4) = [x_{i+1}]$, where $i$ is the smallest index in $K_P$. Recall that the edges of $H = G[\{z_1, z_2\}\cup [x_i]]$ separate the plane into the five above-described regions $R_1, \ldots, R_5$. Again, we can use $z_3$ or $z_4$ to prove that $C_2$ is either in $R_2$ or in $R_4$ (say $R_4$). Therefore, $x_{i+2}$ is in $R_4$ which implies $z_3$ and $z_4$ are also in $R_4$. Thus, $z_1$ is not adjacent to $z_3$ nor $z_4$. Furthermore, $z_2$ cannot be adjacent to both $z_3$ and $z_4$ for otherwise we can obtain a $K_5$ in the subgraph of $G'$ induced by $[x_{i+1}]\cup \{z_2, z_3, z_4\}$ by contracting the edge $z_4x_{i+2}$ (see Figure~\ref{K5minor}). Thus, without loss of generality, suppose $z_2$ and $z_3$ are not adjacent.
\begin{figure}[!hbt]
\begin{center}
\scalebox{.7}{\input{K5minor.pspdftex}}
\caption{$K_5$ minor of $G'$ is obtained by contracting $z_4x_{i+2}$.}\label{K5minor}
 \end{center}
\end{figure}

 Consider the path $P'$ obtained from $P$ by replacing $x_{i+1}$ with $x'_{i+1} = z_3$. Since no edge in $H$ is crossed,  for any $v\in V(G)\setminus (\{z_1, z_2\}\cup [x_i])$,  if $v$ is adjacent to $x_{i-1}$ then it is not in $R_4$, and if $v$ is adjacent to  $z_3$ then it must be in $R_4$. Since neither $z_1$ nor $z_2$ are adjacent to $z_3$ and $x_{i+1}$ is not adjacent to $x_{i-1}$, there is no vertex in $Z_{P'}$ adjacent to $\{x_{i-1}, x_i, z_3\}$. This implies that if $k\in K_{P'}$, then either $k\le i-2$ or $k\ge i+1$. Note that if $k\in K_{P'}$ and  $k \le i-2$, then $k\in K_P$. This implies that the minimum index in $K_{P'}$ is strictly greater than $i$.  Hence by our choice of $P$, the path $P'$ is not semi-nice, that is $J_{P'} \neq \emptyset$.

 Observe that if $j\in J_{P'}$, then either $j\le i-2$ or $j\ge i+2$. Note that if $j\in J_{P'}$ and  either $j \le i-2$ or $j\ge i+4$, then $j\in J_P$.  Since $J_P$ is empty, then $J_{P'}\subseteq \{i+2, i+3\}$.
Let $z'_1,z'_2\in Z_{P'}$  be such that $N_{P'}(z'_1) = N_{P'}(z'_2) = [x'_j]$, for some $j\in J_{P'}$ with $J_{P'}\subseteq \{i+2, i+3\}$. Note that for the two possible values of $j$, both $z'_1$ and $z'_2$ are adjacent to $x_{i+3}$. Since $P$ is a shortest $(C_1,C_2)$-path,  neither $z_2$ nor $x_{i+1}$ are adjacent to $x_{i+3}$ and therefore $z'_1$ and $z'_2$ are in $R_4$. Let $R'_1$ be the region bounded by $x'_{j-1}x'_jz'_1$ not containing the vertex $z'_2$ and  $R'_3$ be the region bounded by $x'_{j-1}x'_jz'_2$ not containing the vertex $z'_1$. Both of these regions are contained in $R_4$. With the same argument used above in the proof of existence of a semi-nice path, one shows that if $j\in J_{P'}$, then $C_1$ is either contained in $R'_1$ or in $R'_3$. We get a contradiction as the path $P$ from $V(C_1)$ to $x_{i-1}$ crosses an edge of $H$. 
\end{proof}

\begin{claim}\label{greatpath}
There exists an induced path $Q = x_0 x_1 \cdots x_p x_{p+1}$ with the following properties:
{\renewcommand{\theenumi}{P${}_\arabic{enumi}$}
\begin{enumerate}
\item \label{q1} $P = x_1 \cdots x_p$ is a shortest $(C_1,C_2)$-path and is a nice path;
\item \label{q2} $x_0\in V(C_1)$ and $x_{p+1}\in V(C_2)$ but $x_0 x_1$ and $x_p x_{p+1}$ are not crossed edges; and
\item \label{q3} there is at most one vertex in $Z_Q$ adjacent to both vertices in $\{x_0, x_3\}$ and at most one vertex in $Z_Q$ adjacent to both vertices in $\{x_{p-2}, x_{p+1}\}$.
\item \label{q4} for any $i<j$, if there are two vertices adjacent to $[v_i]$ and two vertices adjacent to $[v_j]$,  then the number of vertices adjacent to $[v_{i+1}]$ or to $[v_{j-1}]$ is at most $1$.
\end{enumerate}
}
\end{claim}

\begin{proof}
By Claim~\ref{shortestnicepath} there exists a shortest $(C_1,C_2)$-path $P = x_1 \cdots x_p$ which is nice. Without loss of generality, we may assume that $x_1=v_1$ and $x_p=v_2$.
According to Claim~\ref{nodiag}, we can choose vertices $x_0\in \{u_1, t_1\}$ and $x_{p+1}\in \{u_2, t_2\}$ such that $Q$ is induced. Therefore, we have at least one path satisfying properties \ref{q1} and \ref{q2}. We say that $x_0$ is a {\it valid endpoint} if there is at most one vertex in $Z_Q$ adjacent to both vertices in  $\{x_0, x_3\}$ and $x_{p+1}$ is a {\it valid endpoint} if there is at most one vertex in $Z_Q$ adjacent to both vertices in  $\{x_{p-2}, x_{p+1}\}$.

Let $Q$ be a path satisfying properties \ref{q1} and \ref{q2} which maximizes the number of valid endpoints of $Q$.
 
 \vspace{12pt}
 
 Let us first show that $Q$ has only valid endpoints, and satisfies  property~\ref{q4}. By contradiction, suppose that $Q$ has an invalid endpoint. Without loss of generality, $x_0$ is invalid. 
 
Let $z_1, z_2\in Z_Q$ be two vertices adjacent to both vertices in $\{x_0,  x_3\}$.  
 Since $P$ is a shortest $(C_1,C_2)$-path, no vertex of $C_1$ is adjacent to $x_3$. Therefore, no edge of 
 $x_0 x_1 x_2 x_3 z_1$ and $x_0 x_1 x_2 x_3 z_2$  is crossed.
Let $R_1$ be the region bounded by $x_0 x_1 x_2 x_3 z_1$ that does not contain $z_2$ and  $R_2$ be the region bounded by $x_0 x_1 x_2 x_3 z_2$ that does not contain $z_1$. Since the edges bounding the regions $R_1$ and $R_2$ are not crossed, then the crossing $C_1$ is contained in one of the regions $R_1$ or $R_2$ (say $R_1$). Let $\hat x_0$ be the vertex of $\{u_1,t_1\}\setminus \{x_0\}$ (see Figure~\ref{invalid}). 

\begin{figure}[!hbt]
\begin{center}
\scalebox{.8}{\input{invalid.pspdftex}}
\caption{Regions $R_1$ and $R_2$ and the vertex $\hat x_0$.}\label{invalid}
 \end{center}
\end{figure}
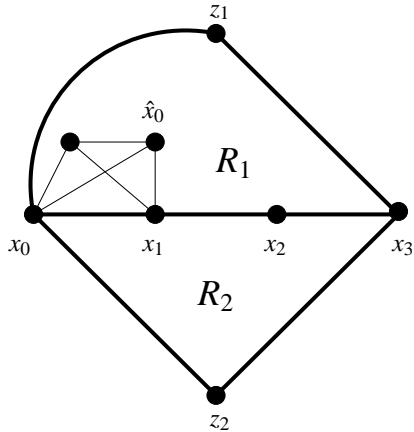

Assume first that  $\hat x_0$ is not adjacent to $x_2$.
Let $\hat Q$ be the path obtained from $Q$ by replacing $x_0$ with $\hat x_0$. 
Clearly the path $\hat Q$ is induced and satisfies properties \ref{q1} and \ref{q2}. 
By definition of $Q$, $\hat x_0$ must be an invalid endpoint.
Hence, there is a vertex $\hat z$ in $Z_{\hat Q}\setminus \{z_1\}$  which is adjacent to $\hat x_0$ and $x_3$.
This vertex in necessarily inside $R_1$ because it is adjacent to $x_0$.
But then, by planarity, $z_1$ cannot be adjacent to $x_1$ and $x_2$, a contradiction to $z_1\in Z_Q$.

Assume now that $\hat x_0$ is adjacent to $x_2$. Let $Q'$ be the path obtained from $Q$ by replacing $x_0$ with $w_1$ and $x_1$ with $\hat x_0$. Note that $Q'$ is induced as $w_1$ is not adjacent to $x_2$ by Claim~\ref{nodiag}. 

Note that property \ref{q2} is valid for $Q'$. The path $P' = \hat x_0 x_2 \cdots x_p$ is a $(C_1,C_2)$ shortest path. 
Let us prove that $P'$ is nice and so that $P'$ satisfies property \ref{q1}. 
If $p=3$, then, since no vertex in the cluster of $C_1$ is adjacent to $x_3$, at most two vertices are in $Z_{P'}$ for otherwise we would get a $K_{3,3}$ in $G-\{w_1,w_2\}$, which is impossible as this graph is planar. Thus $P'$ is nice. Suppose now that $p\ge 4$.  By planarity,  $z_1$ is not adjacent to $x_1$, so $z_1$ is adjacent to $x_2$ as $z_1\in Z_Q$. In addition, $z_1x_2$ is contained in $R_1$. Thus, any vertex in $Z_{P'}$ adjacent to $\hat x_0$ must be in region $R_1$ and cannot be adjacent to $x_3$. Hence no vertex is adjacent to $[x_2]_{P'}$ so, since $P$ is a nice path,  $P'$ is also a nice path.

By definition of $Q$, $w_1$ must be an invalid endpoint of $Q'$. 
Hence, there is a vertex $z'$ in $Z_{Q'}\setminus \{z_1\}$  which is adjacent to $w_1$ and $x_3$.
This vertex in necessarily inside $R_1$ because neither $x_0$ nor $x_1$ are  adjacent to $x_3$.
But then, by planarity, $z_1$ cannot be adjacent to $x_1$ and $x_2$, a contradiction to $z_1\in Z_Q$.

\vspace{12pt}

Let us now prove that $Q$ satisfies property~\ref{q4}.
By contradiction, suppose $Q$ does not. Let $z_1, z_2, z'_1, z'_2 \in Z_Q$ be such that both $z_1$ and $z_2$ are adjacent to $[x_i]$ and $z'_1$ and $z'_2$ are adjacent to $[x_j]$. Consider the regions $R_1, \ldots, R_5$ related to $z_1$ and $z_2$ used in Figure \ref{sameneigh}.  Consider the regions $R'_1, \ldots, R'_5$ related to $z'_1$ and $z'_2$ used in Figure \ref{sameneigh} for $i=j$.

Let $z\in Z_Q$ be adjacent to $[x_{i+1}]$. Note that we can have $\{z_1, z_2\} \cap \{u_1, t_1\} \neq \emptyset$ if $i=1$. But since $dist(C_1, C_2)\ge 2$, the edges $z_1x_{i+1}$ and $z_2x_{i+1}$ are not crossed. Furthermore, since no vertex in the cluster of $C_1$ is adjacent to $x_3$ and not vertex in the cluster of $C_2$ is adjacent to $x_1$ ($P$ is a shortest $(C_1,C_2)$-path), then $z$ is not in the cluster of either crossing.

Therefore, since $z$ is adjacent to both $x_i$ and $x_{i+2}$, we must have that both $z$ and $x_3$ are in $R_2$ or in $R_4$ (say $R_2$). This also implies that $C_2$ is in $R_2$. Note also that, by our choice of $x_0$, the edges $z_1x_i$ and $z_2 x_i$ are not crossed. Therefore, $C_1$ is contained in $R_1\cup R_3\cup R_5$. With a symmetric argument, we have that $C_1$ is either in $R'_1$ or in $R'_3$ (say $R_1$). Since both $z'_1$ and $z'_2$ are also in $R_2$, then $R'_1\cup R'_3$ are contained in $R_2$ and we get a contradiction.
\end{proof}

Let $Q$ be a path given by Claim~\ref{greatpath}.
Without loss of generality, suppose $x_1 = v_1$ and $x_p = v_2$. Note also that Claim~\ref{nodiag} implies $w_1$ and $w_2$ are not in $Z_Q$ and therefore $G[V(Q)\cup Z_Q]$ is planar.

\begin{claim}\label{dist>2}
 $dist(C_1,C_2)=2$ and there is a vertex adjacent to $x_0$ and $x_4$.
\end{claim}
\begin{proof}
Suppose not. Then no vertex in $Z_Q$ is adjacent to vertices at distance at least four in $Q$. 
Observe that this is the case when $dist(C_1,C_2)\geq 3$, since $x_1\dots x_p$ is a shortest $(C_1,C_2)$-path.

Since $P$ is a nice and shortest $(C_1,C_2)$-path, then the only vertices in $Z_Q$ adjacent to vertices at distance at least three in $Q$ must be adjacent to both $x_0$ and $x_3$ or to both $x_{p-2}$ and $x_{p+1}$.
 By the property~\ref{q3} of Claim~\ref{greatpath}, there is at most  one vertex, say $z$,  adjacent to $x_0$ and $x_3$
 and at most one vertex, say $z'$, adjacent to $x_{p-2}$ and $x_{p+1}$.

Let us make few observations.
\begin{itemize}
\item[Obs. 1]
If two vertices $z_1$ and $z_2$ distinct from $z$ are adjacent to $[x_2]$, then no vertex is adjacent to $[x_1]$ and $N_Q(z) = \{x_0, x_1, x_3\}$.
Indeed $z$ must be in the region $R_5$ in Figure \ref{sameneigh} because it is adjacent to $x_0$ and $x_3$. 
By the planarity of $G[V(Q)\cup Z_Q]$ and since $z$ is adjacent to $x_0$, $x_0$ must also be in $R_5$. 
Again by planarity, $z$ is not adjacent to $x_2$ and, therefore, must be adjacent to $x_1$ as $z\in Z_Q$. 

\item[Obs. 2]
If two vertices $z_1$ and $z_2$  distinct from $z$ are adjacent to $[x_1]$, then no vertex is adjacent to $[x_2]$ and $N_Q(z) = \{x_0, x_2, x_3\}$. This argument is symmetric to Observation 1.

\end{itemize}

Suppose that  $z$ exists.\\
If $z'$ exists, by Observations 1 and 2 (and their analog for $z'$) and the properties of $Q$ from Claim~\ref{greatpath}, the path $Q$ is good in $G-z'$ because
it is great in $G-\{z, z'\}$. Hence  $Q$ is magic, a contradiction to Claim~\ref{nomagic}.
Hence $z'$ does not exists.\\ 
By Claim~\ref{nodiag}, $w_2$ is not adjacent to $x_{p-1}$ and $w_1$ is not adjacent to $x_{p}$ since
$dist(C_1,C_2)\geq 2$. So, by planarity of $G-\{w_1,w_2\}$, at most two vertices are adjacent to $[x_p]$.
Let $y$ be a vertex adjacent to $[x_p]$.
The path $Q$ is  not great in $G-\{y,z\}$, for otherwise it would be magic.
Hence, according to the properties of $Q$ and the above observations, there must be two vertices adjacent to $[x_p]$, two vertices adjacent to $[x_{p-1}]$ and one vertex adjacent to $[x_{p-2}]$.
Let $z_1$ and $z_2$ be the two vertices adjacent to $[x_{p-1}]$ and $R_1 \dots R_5$ be the regions as 
in Figure~\ref{sameneigh} with $i=p-1$.
Since there is a vertex adjacent to $[x_{p-2}]$, then $C_1$ is in $R_1$ or $R_3$, and
$C_2$ is in $R_2$ or $R_4$ because a vertex is adjacent to $[x_p]$.
But by Claim~\ref{sep4} the $4$-cycle $z_1x_pz_2x_{p-2}$ is not nicely separating, so there is no vertex inside $R_5$.
Since $G$ is triangulated, and $x_{p-2}x_p$ is not an edge because $P$ is a shortest $(C_1,C_2)$-path,
$z_1z_2\in E(G)$.
Now the path $Q$ is good in $G - \{z_1,z_2\}$ and so is magic. This contradicts Claim~\ref{nomagic}.

Hence we may assume that $z$ does not exists and by symmetry that $z'$ does not exist. We get a contradiction similarly by considering a vertex $w$ adjacent to $[x_1]$ in place of $z$.
\end{proof}

\begin{claim}\label{unseul}
There is precisely one vertex $z\in Z_Q$ adjacent to both $x_0$ and $x_4$. 
\end{claim}
\begin{proof}
Observe that there are at most two vertices adjacent to $x_0$ and $x_4$.
Indeed such vertices cannot be in the crossings because $dist(C_1,C_2)=2$. Thus if there were three such vertices, together with contracting  the path $x_1 x_2 x_3$ we would get $K_{3,3}$ minor in $G-\{w_1,w_2\}$, a contradiction.

Suppose by contradiction that two distinct vertices $z_1, z_2\in Z_Q$ adjacent to vertices $x_0$ and $x_4$.  The edges of $Q$ are contained in the same region of the plane bounded by the cycle $x_0 z_1 x_4 z_2$. Therefore, both crossings are also in the region containing the edges of $Q$. By Claim~\ref{4cycle}, the region bounded by the cycle $x_0 z_1 x_4 z_2$ that does not contain the crossings has no vertex in its interior. Since $G$ is triangulated, $z_1z_2\in E(G)$ as $x_0$ because $x_4$ are not adjacent as $dist(C_1,C_2)=2$.

By the property \ref{q3} of Claim~\ref{greatpath}, $z_1$ and $z_2$ cannot be both adjacent to the five vertices in $Q$. Therefore, without loss of generality, suppose $|N_Q(z_2)| \le 4$. 
Let us prove that $Q$ is great in $H=(G-z_2) \setminus \{z_1x_0,z_1x_4\}$. 

\begin{itemize}
\item[(i)] If a vertex $t$ in $G-\{z_1,z_2\}$ is adjacent to at least four vertices of $Q$, then without loss of generality it is adjacent
to $\{x_0, x_1, x_2, x_3\}$ as it cannot be adjacent to $x_0$ and $x_4$.
Now by property \ref{q3}, $z_1$ and $z_2$ are not adjacent to $x_3$.
Hence one of them (the one such that $x_0x_1x_2x_3x_4z_i$ separates $t$ from $z_{3-i}$) cannot be adjacent to
any vertex of   $\{x_1, x_2, x_3\}$, a contradiction to the fact that it is in $Z_Q$. Hence $Q$ satisfies (a) in $H$. 

\item[(ii)] If two vertices $t_1$ and $t_2$ of $H$ are adjacent to $[x_2]$, then necessarily $x_1t_1x_2t_2$ is a nicely separating, a contradiction to Claim~\ref{sep4}. Hence there is at most one vertex  of $H$ adjacent to $[x_2]$.
Thus  $Q$ satisfies (b) in $H$. 

\item[(iii)]  If two vertices $r_1$ and $r_2$ of $H$ are adjacent to $[x_1]$, then no vertex is adjacent to $[x_2]$.
Indeed suppose for a contradiction that a vertex $t$ is adjacent to $[v_2]$
none of $\{r_1,r_2, t\}$  is in $\{w_1,w_2\}$ by Claim~\ref{nodiag} and because $dist(C_1,C_2)\geq 2$.
Now contracting the path $tx_3x_4z_2$ into a vertex $w$, we obtain a $K_{3,3}$ with parts $\{r_1, r_2, w\}$ and $\{x_0,x_1,x_2\}$. This contradicts the planarity of $G$.

Symmetrically,  if two vertices of $H$ are adjacent to $[x_3]$, then no vertex is adjacent to $[x_2]$.
Therefore   $Q$ satisfies (c) in $H$. 
\end{itemize}

It follows that $Q$ is a good path in $H' = (G-z_2)\setminus z_1 x_4$. 
Let $\phi$ be a safe $L$-colouring of $Q$ in $H'$ obtained by Lemma~\ref{goodpath}. Since $Q$ meets the two crossings, $G-Q$ is planar. Furthermore, $G-Q$ can be drawn in the plane such that all vertices on the outer face are those in $N(Q)$. 
Every vertex of $Z_Q\setminus \{z_1,z_2\}$ is safe in $H'$ and so in $G$, so $|L_\phi(v)|\ge 3$.
In $H'$, $z_1$ is safe and in $G$, $z_1$ has one more neighbour in $Q$ in $G$ than $H'$, namely $x_4$.
Thus in $G$, $|L_{\phi}(z_1)|\geq 2$ because $z_1$ was safe in $H'$.
Since $z_2$ has at most four neighbours in $Q$, we have $|L_\phi(z_2)| \ge 1$. Now $z_1$ is adjacent to $z_2$, so $L_\phi$ is a $\{z_1,z_2\}$-suitable assignment for $G-Q$. Hence by Theorem~\ref{suitable}, $G-Q$ is $L_\phi$-colourable and so $G$ is $L$-colourable, a contradiction.
\end{proof}

\begin{itemize}

\item Assume first that $|N_Q(z)|=5$.
Let $H=G\setminus \{zx_0, zx_4\}$.
$z$ is the unique vertex adjacent to $x_0$ and $x_4$. Moreover by
property \ref{q3} $z$ is the unique vertex adjacent to $x_0$ and $x_3$ and the unique one adjacent to
$x_1$ and $x_4$.
Hence $Q$ satisfies (a) in $H$.
Moreover, for $1\leq i\leq 3$, there is at most one vertex distinct form $z$ adjacent to $[x_i]$ otherwise $G[V(Q)\cup Z_Q]$ would contain a $K_{3,3}$.
Hence $Q$ also satisfies (b) and (c) in $H$.
Therefore $Q$ is great in $H$.
 By Lemma~\ref{nicepath}, there exists a safe $L$-colouring $\phi$ of $Q$ in $H$. 
 Thus in $G$, every vertex in $Z_Q\setminus \{z\}$ satisfies $|L_\phi(v)|\ge 3$ while  $|L_\phi(z)|\ge 1$.
 Hence $L_\phi$ is suitable for $G-Q$. Therefore, by Theorem~\ref{suitable}, $G-Q$ is $L_\phi$-colourable and so $G$ is $L$-colourable, a contradiction.

\item Assume now that $|N_Q(z)|\leq 4$.

Suppose that there are two distinct vertices $z_1, z_2\in Z_Q$ with $z_1$ adjacent to $x_0$ and $x_3$ and $z_2$ adjacent to $x_1$ and $x_4$. Let $R_1$ be the region bounded by the cycle $x_0 x_1 x_2 x_3 z_1$ not containing $z_2$ and $R_2$ be the region bounded by the cycle $x_1 x_2 x_3 x_4 z_2$ not containing $z_1$ (see Figure~\ref{R1R2}). Now, note that any vertex adjacent to both $x_0$ and $x_4$ is not in $R_1\cup R_2$ and any vertex adjacent to $x_2$ must be in $R_1\cup R_2$. Therefore, $z\in \{z_1, z_2\}$. Indeed if this was not true, then by property~\ref{q3}  $z$ is not adjacent to $x_1$ nor $x_3$. Thus  $z$ must be adjacent to $x_2$ as it is in $Z_Q$. So $z$ is inside $R_1\cup R_2$, which contradicts the fact that it is adjacent to $x_0$ and $x_4$.
\begin{figure}[!hbt]
\begin{center}
\scalebox{.8}{\input{R1R2.pspdftex}}
\caption{Regions $R_1$ and $R_2$.}\label{R1R2}
 \end{center}
\end{figure}
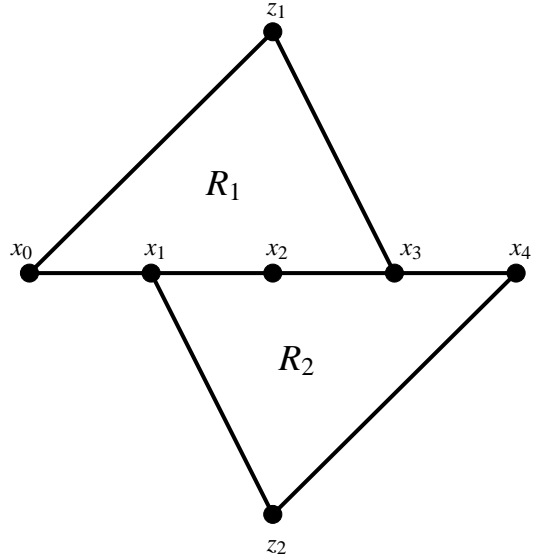

Thus, at most one other vertex $z'$ in $Z_Q\setminus \{z\}$ is adjacent to vertices at distance three in $Q$.
By symmetry, we may assume that $z'$ is adjacent to $x_0$ and $x_3$.
Hence all vertices  in $Z_Q\setminus \{z,z'\}$ are adjacent to some $[x_i]$ for $1\leq i\leq 3$.
Similarly to (ii) and (iii) in Claim~\ref{unseul}, one shows that
$Q$ also satisfies (a) and (b) in $(G-z)\setminus z'x_0$.
 Hence $Q$ is a good path in $G-z$.  
 Then $Q$ is magic, a contradiction to Claim~\ref{nomagic}.
\end{itemize} 


\end{proof}

\section*{Acknowledgement}
The authors would like to thank Claudia Linhares Sales for stimulating discussions.

\end{document}

%% file: sameneigh.pspdftex
\begin{picture}(0,0)%
\includegraphics{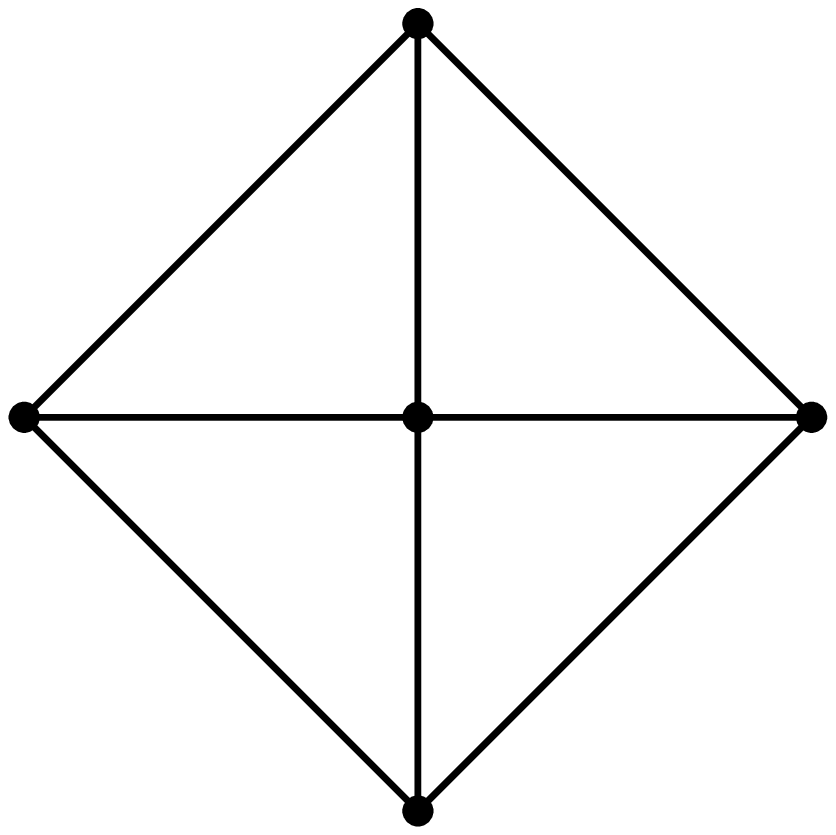}%
\end{picture}%
\setlength{\unitlength}{4144sp}%
\begingroup\makeatletter\ifx\SetFigFont\undefined%
\gdef\SetFigFont#1#2#3#4#5{%
  \reset@font\fontsize{#1}{#2pt}%
  \fontfamily{#3}\fontseries{#4}\fontshape{#5}%
  \selectfont}%
\fi\endgroup%
\begin{picture}(3960,4233)(-2084,-1300)
\put( 46,974){\makebox(0,0)[lb]{\smash{{\SetFigFont{12}{14.4}{\rmdefault}{\mddefault}{\updefault}{\color[rgb]{0,0,0}$x_i$}%
}}}}
\put(-44,2774){\makebox(0,0)[lb]{\smash{{\SetFigFont{12}{14.4}{\rmdefault}{\mddefault}{\updefault}{\color[rgb]{0,0,0}$z_1$}%
}}}}
\put(-44,-1231){\makebox(0,0)[lb]{\smash{{\SetFigFont{12}{14.4}{\rmdefault}{\mddefault}{\updefault}{\color[rgb]{0,0,0}$z_2$}%
}}}}
\put(1756,974){\makebox(0,0)[lb]{\smash{{\SetFigFont{12}{14.4}{\rmdefault}{\mddefault}{\updefault}{\color[rgb]{0,0,0}$x_{i+1}$}%
}}}}
\put(-2069,974){\makebox(0,0)[lb]{\smash{{\SetFigFont{12}{14.4}{\rmdefault}{\mddefault}{\updefault}{\color[rgb]{0,0,0}$x_{i-1}$}%
}}}}
\put(-1304,1919){\makebox(0,0)[lb]{\smash{{\SetFigFont{14}{16.8}{\rmdefault}{\mddefault}{\updefault}{\color[rgb]{0,0,0}{\Large $R_5$}}%
}}}}
\put(-719,1334){\makebox(0,0)[lb]{\smash{{\SetFigFont{14}{16.8}{\rmdefault}{\mddefault}{\updefault}{\color[rgb]{0,0,0}{\Large $R_1$}}%
}}}}
\put(406,1334){\makebox(0,0)[lb]{\smash{{\SetFigFont{14}{16.8}{\rmdefault}{\mddefault}{\updefault}{\color[rgb]{0,0,0}{\Large $R_2$}}%
}}}}
\put(-719,209){\makebox(0,0)[lb]{\smash{{\SetFigFont{14}{16.8}{\rmdefault}{\mddefault}{\updefault}{\color[rgb]{0,0,0}{\Large $R_3$}}%
}}}}
\put(406,209){\makebox(0,0)[lb]{\smash{{\SetFigFont{14}{16.8}{\rmdefault}{\mddefault}{\updefault}{\color[rgb]{0,0,0}{\Large $R_4$}}%
}}}}
\end{picture}%

%% file: K5minor.pspdftex
\begin{picture}(0,0)%
\includegraphics{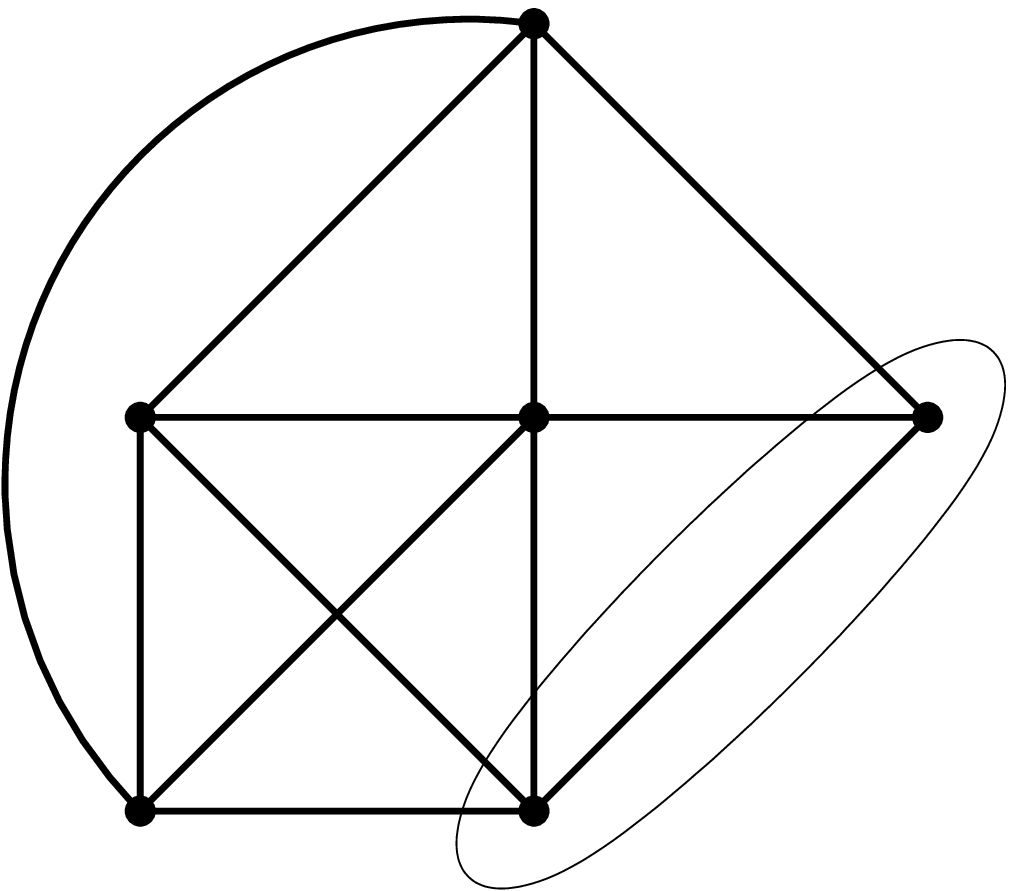}%
\end{picture}%
\setlength{\unitlength}{4144sp}%
\begingroup\makeatletter\ifx\SetFigFont\undefined%
\gdef\SetFigFont#1#2#3#4#5{%
  \reset@font\fontsize{#1}{#2pt}%
  \fontfamily{#3}\fontseries{#4}\fontshape{#5}%
  \selectfont}%
\fi\endgroup%
\begin{picture}(4608,4260)(-641,-1327)
\put(-359,974){\makebox(0,0)[lb]{\smash{{\SetFigFont{12}{14.4}{\rmdefault}{\mddefault}{\updefault}{\color[rgb]{0,0,0}$x_i$}%
}}}}
\put(-224,-1231){\makebox(0,0)[lb]{\smash{{\SetFigFont{12}{14.4}{\rmdefault}{\mddefault}{\updefault}{\color[rgb]{0,0,0}$z_2$}%
}}}}
\put(1846,974){\makebox(0,0)[lb]{\smash{{\SetFigFont{12}{14.4}{\rmdefault}{\mddefault}{\updefault}{\color[rgb]{0,0,0}$x_{i+1}$}%
}}}}
\put(1756,2774){\makebox(0,0)[lb]{\smash{{\SetFigFont{12}{14.4}{\rmdefault}{\mddefault}{\updefault}{\color[rgb]{0,0,0}$z_3$}%
}}}}
\put(1711,-1231){\makebox(0,0)[lb]{\smash{{\SetFigFont{12}{14.4}{\rmdefault}{\mddefault}{\updefault}{\color[rgb]{0,0,0}$z_4$}%
}}}}
\put(3511,524){\makebox(0,0)[lb]{\smash{{\SetFigFont{12}{14.4}{\rmdefault}{\mddefault}{\updefault}{\color[rgb]{0,0,0}$x_{i+2}$}%
}}}}
\end{picture}%

%% file: invalid.pspdftex
\begin{picture}(0,0)%
\includegraphics{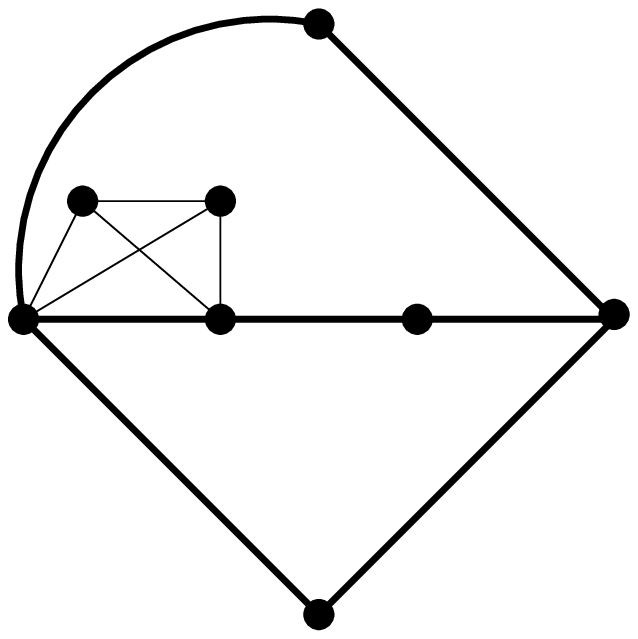}%
\end{picture}%
\setlength{\unitlength}{4144sp}%
\begingroup\makeatletter\ifx\SetFigFont\undefined%
\gdef\SetFigFont#1#2#3#4#5{%
  \reset@font\fontsize{#1}{#2pt}%
  \fontfamily{#3}\fontseries{#4}\fontshape{#5}%
  \selectfont}%
\fi\endgroup%
\begin{picture}(2970,3288)(-1544,-805)
\put(-44,2324){\makebox(0,0)[lb]{\smash{{\SetFigFont{12}{14.4}{\rmdefault}{\mddefault}{\updefault}{\color[rgb]{0,0,0}$z_1$}%
}}}}
\put(-44,-736){\makebox(0,0)[lb]{\smash{{\SetFigFont{12}{14.4}{\rmdefault}{\mddefault}{\updefault}{\color[rgb]{0,0,0}$z_2$}%
}}}}
\put(-134,164){\makebox(0,0)[lb]{\smash{{\SetFigFont{14}{16.8}{\rmdefault}{\mddefault}{\updefault}{\color[rgb]{0,0,0}{\Large $R_2$}}%
}}}}
\put(-1529,569){\makebox(0,0)[lb]{\smash{{\SetFigFont{12}{14.4}{\rmdefault}{\mddefault}{\updefault}{\color[rgb]{0,0,0}$x_0$}%
}}}}
\put(-539,569){\makebox(0,0)[lb]{\smash{{\SetFigFont{12}{14.4}{\rmdefault}{\mddefault}{\updefault}{\color[rgb]{0,0,0}$x_1$}%
}}}}
\put(361,569){\makebox(0,0)[lb]{\smash{{\SetFigFont{12}{14.4}{\rmdefault}{\mddefault}{\updefault}{\color[rgb]{0,0,0}$x_2$}%
}}}}
\put(1306,569){\makebox(0,0)[lb]{\smash{{\SetFigFont{12}{14.4}{\rmdefault}{\mddefault}{\updefault}{\color[rgb]{0,0,0}$x_3$}%
}}}}
\put(-539,1559){\makebox(0,0)[lb]{\smash{{\SetFigFont{12}{14.4}{\rmdefault}{\mddefault}{\updefault}{\color[rgb]{0,0,0}$\hat x_0$}%
}}}}
\put(  1,1154){\makebox(0,0)[lb]{\smash{{\SetFigFont{14}{16.8}{\rmdefault}{\mddefault}{\updefault}{\color[rgb]{0,0,0}{\Large $R_1$}}%
}}}}
\end{picture}%

%% file: R1R2.pspdftex
\begin{picture}(0,0)%
\includegraphics{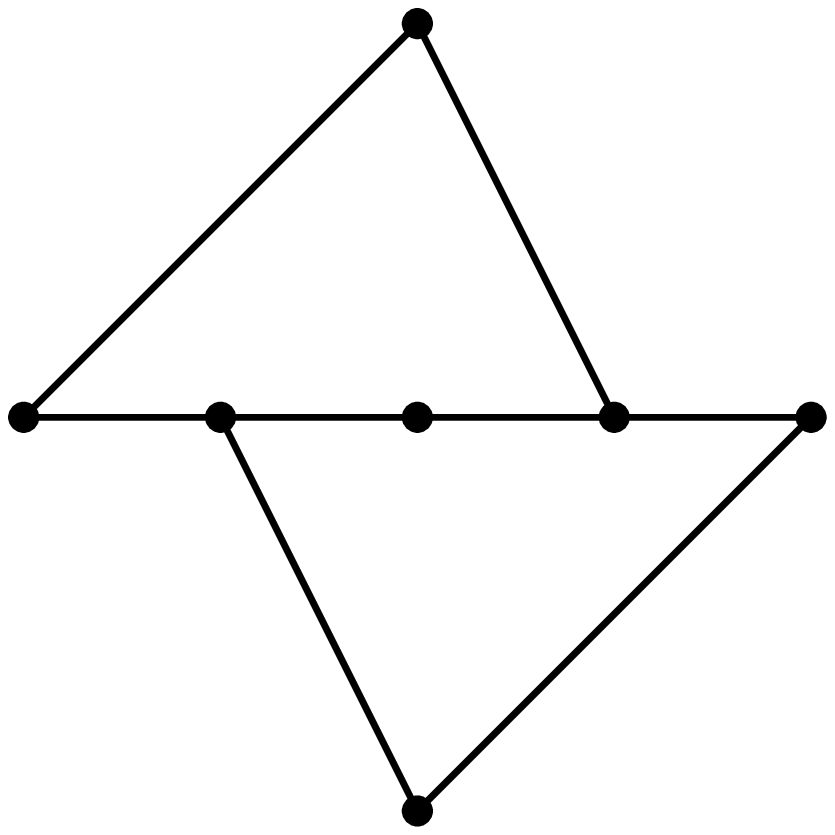}%
\end{picture}%
\setlength{\unitlength}{4144sp}%
\begingroup\makeatletter\ifx\SetFigFont\undefined%
\gdef\SetFigFont#1#2#3#4#5{%
  \reset@font\fontsize{#1}{#2pt}%
  \fontfamily{#3}\fontseries{#4}\fontshape{#5}%
  \selectfont}%
\fi\endgroup%
\begin{picture}(3825,4233)(-1949,-1300)
\put(-44,2774){\makebox(0,0)[lb]{\smash{{\SetFigFont{12}{14.4}{\rmdefault}{\mddefault}{\updefault}{\color[rgb]{0,0,0}$z_1$}%
}}}}
\put(-44,-1231){\makebox(0,0)[lb]{\smash{{\SetFigFont{12}{14.4}{\rmdefault}{\mddefault}{\updefault}{\color[rgb]{0,0,0}$z_2$}%
}}}}
\put(1756,974){\makebox(0,0)[lb]{\smash{{\SetFigFont{12}{14.4}{\rmdefault}{\mddefault}{\updefault}{\color[rgb]{0,0,0}$x_4$}%
}}}}
\put(946,974){\makebox(0,0)[lb]{\smash{{\SetFigFont{12}{14.4}{\rmdefault}{\mddefault}{\updefault}{\color[rgb]{0,0,0}$x_3$}%
}}}}
\put(-44,974){\makebox(0,0)[lb]{\smash{{\SetFigFont{12}{14.4}{\rmdefault}{\mddefault}{\updefault}{\color[rgb]{0,0,0}$x_2$}%
}}}}
\put(-944,974){\makebox(0,0)[lb]{\smash{{\SetFigFont{12}{14.4}{\rmdefault}{\mddefault}{\updefault}{\color[rgb]{0,0,0}$x_1$}%
}}}}
\put(-1934,974){\makebox(0,0)[lb]{\smash{{\SetFigFont{12}{14.4}{\rmdefault}{\mddefault}{\updefault}{\color[rgb]{0,0,0}$x_0$}%
}}}}
\put(-494,1424){\makebox(0,0)[lb]{\smash{{\SetFigFont{14}{16.8}{\rmdefault}{\mddefault}{\updefault}{\color[rgb]{0,0,0}{\Large $R_1$}}%
}}}}
\put( 46,119){\makebox(0,0)[lb]{\smash{{\SetFigFont{14}{16.8}{\rmdefault}{\mddefault}{\updefault}{\color[rgb]{0,0,0}{\Large $R_2$}}%
}}}}
\end{picture}%

%% file: ch2cross.bbl
\begin{thebibliography}{99}


\bibitem{Alb08}
M.~O.~Albertson.
Chromatic Number, Independence Ratio, and Crossing Number.
{\it Ars Mathematica Contemporanea}~ 1:1--6, 2008.
 
 
\bibitem{AHMW}
M. O. Albertson, M. Heenehan, A. McDonough, and J. Wise.
Coloring graphs with given crossing patterns.
{\it manuscript}.

 \bibitem{BaTo10} 
J. Bar\'at and G. T\'oth.
Towards the Albertson Conjecture.
\newblock {\it Electronic Journal of Combinatorics}~17: R-73, 2010.
 
\bibitem{DLR}
Z. Dvo\v{r}\'ak, B. Lidick\'y, and R. \v{S}krekovski.
\newblock Graphs with two crossings are 5-choosable. 
\newblock (arXiv:1103.1801v1 [math.CO]).


\bibitem{EHLP11}
 R. Erman, F. Havet, B. Lidicky, and O. Pangrac. 
\newblock 5-colouring graphs with 4 crossings.
\newblock {\em SIAM J. Discrete Math.}~25(1):401--422, 2011. 



\bibitem{Kur30}
C. Kuratowski. 
\newblock Sur le probl\`eme des courbes gauches en topologie.
\newblock {\it Fund. Math.}~15: 271--283, 1930.


\bibitem{OpZh} 
B. Oporowski and D. Zhao.
Coloring graphs with crossing. 
{\it Discrete Mathematics}~309: 2948--2951, 2009.


\bibitem{Sch}
 M. Schaefer.
 {\it personal communication to M. O. Albertson}.



\bibitem{Tho81}
C. Thomassen. 
\newblock Kuratowski's theorem. 
\newblock {\it J. Graph Theory}~5:225--241, 1981.


\bibitem{Tho94} 
C. Thomassen.
Every planar graph is 5-choosable.
{\it J. Comb. Theory B}~62:180--181, 1994.





\end{thebibliography}
